%% file: _0NP-Bayes-GSSM.tex
\documentclass[12pt,a4paper,fleqn]{article}
\pdfoutput=1

\usepackage[english]{babel}
\RequirePackage{ucs}
\RequirePackage[utf8x]{inputenc}
\RequirePackage[T1]{fontenc}

\RequirePackage{amsthm,amsmath}
\RequirePackage{amssymb,amstext}
\RequirePackage{amsfonts}
\RequirePackage{amsbsy}
\RequirePackage{mathrsfs}
\RequirePackage{bm}
\RequirePackage{bbm}

\usepackage[pdfpagemode=UseOutlines ,plainpages=false
,hypertexnames=false ,pdfpagelabels ,hyperindex=true,colorlinks=true]{hyperref}
	\makeatletter
	\Hy@breaklinkstrue
	\makeatother
\usepackage{color}
\definecolor{darkred}{rgb}{0.6,0.0,0.1}
\definecolor{darkgreen}{rgb}{0,0.5,0}
\definecolor{darkblue}{rgb}{0,0,0.5}
\hypersetup{colorlinks ,linkcolor=darkblue ,filecolor=darkgreen
,urlcolor=darkblue ,citecolor=black}

\usepackage{url}
\usepackage{verbatim}

\RequirePackage{paralist}
\usepackage{enumitem}

\usepackage{natbib}
\usepackage{hypernat}
\renewcommand{\cite}{\citet}
\usepackage{NP-Bayes-GSSM-abrev-package}

\definecolor{dgreen}{rgb}{0,0.5,0}
\definecolor{dblue}{rgb}{0,0,0.9}
\definecolor{dred}{rgb}{0.6,0.0,0.1}
\definecolor{dgold}{rgb}{0.5,0.3,0.0}
\definecolor{dvio}{rgb}{0.6,0.3,0.5}
\definecolor{gray}{rgb}{0.5,0.5,0.5}
\definecolor{dbraun}{rgb}{.5,0.2,0}

\newcommand{\dr}{\color{dred}}

\newtheoremstyle{mysc}
  {3pt}
  {3pt}
  {\it}
  {}
  {\color{darkred}\sc}
  {.}
  {.5em}
  {}

\newtheoremstyle{myas}
  {3pt}
  {3pt}
  {\it}
  {}
  {\color{darkblue}\sc}
  {.}
  {.5em}
  {}

\newtheoremstyle{myex}
  {10pt}
  {10pt}
  {\it}
  {}
  {\color{darkred}\sc}
  {.}
  {.5em}
  {}

\theoremstyle{mysc}\newtheorem{prop}{Proposition}[section]
\theoremstyle{mysc}\newtheorem{coro}[prop]{Corollary}
\theoremstyle{mysc}\newtheorem{theo}[prop]{Theorem}
\theoremstyle{mysc}\newtheorem{lem}[prop]{Lemma}

\theoremstyle{myas}\newtheorem{ass}{Assumption}

\theoremstyle{myex}\newtheorem{rem}{Remark}
\theoremstyle{myex}\newtheorem{illu}[rem]{Illustration}

\numberwithin{equation}{section}

\makeatletter%
\def\@fnsymbol#1{\ensuremath{\ifcase#1\or * \or 1 \or  2\or 3\or 4\or * \or \star \or  , \or
g\or h\or i\else\@ctrerr\fi}}%
\makeatother%

\author{\begin{minipage}{.28\textwidth}\center
{\sc Jan Johannes}\thanks{Corresponding  author.} \thanks{Ensai, Campus de Ker Lann, Rue Blaise
Pascal - BP 37203, 35172 Bruz Cedex, France  and ISBA/CORE, Universit\'e
catholique de Louvain, Voie du Roman Pays 20, 1348~Louvain-la-Neuve,
Belgium,  e-mail:
\url{jan.johannes@ensai.fr}}\\[.5ex]\small CREST-Ensai and Universit\'e
catholique de Louvain\end{minipage}
\and \begin{minipage}{.28\textwidth}\center {\sc Anna Simoni}\thanks{CREST, 15 Boulevard Gabriel P\'{e}ri, 92240 Malakoff, France,
e-mail: \url{simoni.anna@gmail.com}}\\[.5ex]\small CNRS and CREST\\[3.5ex]\null\end{minipage}\and \begin{minipage}{.28\textwidth}\center
{\sc Rudolf Schenk}\thanks{Rudolf is deceased. This manuscript is
  based on a joint work while he was a Ph.D. student at ISBA, Universit\'e
catholique de Louvain.}\\[.5ex]\small Universit\'e catholique de Louvain\\\null\end{minipage}\\[-6ex]}
\date{}
\title{Adaptive  Bayesian estimation in\\ indirect Gaussian sequence space models}

\begin{document}
\maketitle
\begin{abstract}In an indirect Gaussian sequence space model lower and
  upper bounds are derived for the concentration rate of the posterior
  distribution of the parameter of interest shrinking to the parameter
  value $\TrSo$ that generates the data. While this establishes posterior
  consistency, however, the concentration rate depends on both $\TrSo$
  and a tuning parameter which enters the prior distribution. We first
  provide an oracle optimal choice of the tuning parameter, i.e.,
  optimized for each $\TrSo$ separately. The optimal choice of the
  prior distribution allows us to derive an oracle optimal
  concentration rate of the associated posterior distribution.
Moreover, for a given class of parameters and a suitable choice of the
tuning parameter, we show that the resulting uniform
  concentration rate over the given class is optimal in a minimax
  sense.  Finally, we construct a hierarchical
  prior 
  that is adaptive. This means that, given a parameter $\TrSo$ or a
  class of parameters, respectively, the posterior distribution
  contracts at the oracle rate or at the minimax rate over the
  class. Notably, the hierarchical prior does not depend neither on $\TrSo$ nor
  on the given class. Moreover, convergence of the fully data-driven Bayes
  estimator at the oracle or at the minimax rate is  established.
\end{abstract}
{\footnotesize
\begin{tabbing}
\>oracle optimality,  minimax theory, adaptation.\\[.2ex]
\end{tabbing}}

\input{_1Intro}
\input{_2Model}
\input{_3Rate}
\input{_4Adaptive}

\paragraph{Conclusions and perspectives.}In this paper we have presented a
hierarchical prior leading to a fully-data driven Bayes estimator that is minimax-optimal in an indirect sequence space
model. Obviously, the concentration rate based on a hierarchical prior  in an indirect sequence space model
with additional noise in the eigenvalues  is
 only one amongst the many interesting questions for further research
 and we are currently exploring this topic. Moreover, inspired by the
 specific form of the fully-data driven Bayes estimator, as
 discussed in the last section, we are currently studying the effect
 of different choices for the contrast and the penalty term on the
 properties of the estimator.
\paragraph{Acknowledgements.}
The work of Jan Johannes and Rudolf Schenk was supported by the IAP research network no.\ P7/06 of the
Belgian Government (Belgian Science Policy), by the ``Fonds Sp\'eciaux
de Recherche'' from the Universit\'e catholique de Louvain and by the
ARC contract 12/17-045 of the "Communauté française de Belgique",
granted by the Académie universitaire Louvain. Anna Simoni thanks financial support from ANR-13-BSH1-0004 (IPANEMA) and from the Collaborative Research Center 884.

\setcounter{subsection}{0}
\appendix
\section{Appendix:  Proofs of Section \ref{s:tr}}\label{app:proofs:s:tr}
\input{_Proof_3Rate.tex}
\section{Appendix: Proofs of Section \ref{s:ad}}\label{app:proofs:s:ad}
\input{_Proof_4Adaptive.tex}
\bibliography{NP-Bayes-GSSM}
\end{document}

%% file: _1Intro.tex
\section{Introduction}\label{s:in}
Accounting for the fact that inverse problems are widely used in many fields of science,
there  has been over the last decades a growing interest in statistical
inverse problems (see, \textit{e.g.}, \cite{KorostelevTsybakov1993},
\cite{MairRuymgaart1996}, \cite{EvansStark2002}, \cite{KaipioSomersalo2005}, \cite{BissantzHohageMunkRuymgaart2007} and references therein). 
Mathematical statistics has paid special attention to  oracle or
  minimax optimal nonparametric estimation and adaptation in the framework of
inverse problems (see \cite{EfromovichKoltchinskii2001}, \cite{CavalierGolubevLepskiTsybakov2003},
\cite{Cavalier2008} and \cite{HoffmannReiss2008}, to name but a few). Nonparametric estimation in general requires
to choose a tuning parameter which is challenging in practise. Oracle
and minimax estimation is achieved, respectively, if the tuning
parameter is set to an optimal value which relies either on a
knowledge of the unknown parameter of interest or of certain
characteristics of it (such as smoothness). Since both the parameter and its smoothness are unknown, it
is necessary to design a feasible procedure to select the tuning
parameter that adapts to the unknown underlying function  or to its
regularity and achieves the oracle or  minimax rate. Among the  most
prominent approaches stand without doubts model selection
(cf. \cite{BarronBirgeMassart1999} and its exhaustive discussion in \cite{Massart07}),
Stein's unbiased risk estimation and its extensions (cf. \cite{CavalierGolubevPicardTsybakov2002},
\cite{CavalierGolubevPicardTsybakov2002} or
\cite{CavalierHengartner2005}),
 Lepski's method (see, e.g., \cite{Lepskij1990},
\cite{Birge2001}, \cite{EfromovichKoltchinskii2001} or \cite{Mathe2006})
or combinations of the aforementioned strategies (cf. \cite{GoldenshlugerLepski2011} and \cite{ComteJohannes2012}).
On the other hand side,  it seems natural to adopt a Bayesian point of view where the tuning parameter can be
endowed with a prior. As the theory for a general inverse problem --
with a possibly unknown or noisy operator -- is technically highly
involved, we  consider in this paper as a starting point an indirect Gaussian regression which
  is well known to be equivalent to an indirect Gaussian sequence
  space model (in a \cite{LeCam1964} sense, see, e.g., \cite{BrownLow1996a}
  for the direct case and \cite{Meister2011} for the indirect case).\\
Let $\Hspace$ be the Hilbert space of square summable real valued sequences  endowed with the usual inner
product $\Hskalar$ and associated norm $\Hnorm$. In an indirect Gaussian
sequence space model (iGSSM) one aim is to recover a parameter sequence
$\So=\suite{\So}\in\Hspace$  from a transformed version
$(\Ev_j\So_j)_{j\geq1}$ that is blurred by a Gaussian white
noise. Precisely, an observable  sequence of random variables
$(\ObSo)_{j\geq1}$, $\ObSo$ for short,
obeys an indirect Gaussian sequence space model, if
\begin{equation}
  \label{in:de:mod}
 \ObSo_j = \Ev_j\So_j + \sqrt{\ObSoNoL}\ObSoNo_j ,\qquad j\in\Nz,
\end{equation}
\noindent where $\{\ObSoNo_j\}_{j\geq1}$ are unobservable error terms,
which are independent and standard normally distributed, and
$0<\ObSoNoL<1$ is the noise level. The sequence
$\Ev=\suite{\Ev}$ represents the operator that
transforms the signal $\So$. In the particular case of a constant
sequence $\Ev$ the sequence space model is called direct while it
 is  called an indirect sequence space model if the sequence $\Ev$
 tends to zero. We assume throughout the paper that the sequence is bounded.

In this paper we adopt a Bayesian approach, where the parameter
sequence of interest $\So=(\So_j)_{j\geq1}$ itself is a realisation of a
random variable $\RvSo=(\RvSo_j)_{j\geq1}$ and the observable random
variable $\ObSo=(\ObSo_j)_{j\geq1}$ satisfies
\begin{equation}
  \label{in:de:ob}
 \ObSo_j = \Ev_j\RvSo_j + \sqrt{\ObSoNoL}\ObSoNo_j,\quad j\in\Nz
\end{equation}
with independent and standard normally distributed error terms
$\{\ObSoNo_j\}_{j\geq1}$ and noise level $0<\ObSoNoL<1$. Throughout the paper we assume that random parameters
$\{\RvSo_j\}_{j\geq1}$ and the error terms $\{\ObSoNo_j\}_{j\geq1}$
are independent. Consequently, \eqref{in:de:ob} and a specification of the prior
distribution $P_{\RvSo}$ of $\RvSo$ determine completely the joint distribution of $\ObSo$
and $\RvSo$. For a broader overview on Bayesian procedures we refer the reader to the
monograph  by \cite{Robert2007}.\\
Typical prior specifications studied in the
direct sequence space model literature are compound priors, also known as Sieve priors (see,
e.g., \cite{Zhao2000}, \cite{ShenWasserman2001} or
\cite{ArbelGayraudRousseau2013}, Gaussian series priors (cf. \cite{freedman1999}, \cite{Cox1993} or \cite{Castillo2008}), block
priors (cf. \cite{GaoZhou2014}), countable mixture of normal priors
(cf. \cite{BelitserGhosal2003}) and finite mixtures of normal and
Dirac priors (\textit{e.g.} \cite{Abramovichetal1998}). In the context
of an iGSSM,
\cite{KnapikVanderVaartVanZanten2011} and
\cite{KnapikSzaboVanderVaartVanZanten2014} consider Gaussian series
priors and  continuous mixture of Gaussian series priors, respectively. \\
By considering an iGSSM we derive in this paper theoretical properties  of a Bayes
procedure with a Sieve prior specification from a frequentist point of
view, meaning that there exists a true parameter value $\TrSo =
(\TrSo[j])_{j\geq 1}$ associated with the data generating process of
$\suite{\ObSo}$.  A broader overview of frequentist asymptotic
properties of nonparametric Bayes procedures can be found, for example,
in \cite{GhoshRamamoorthi2003}, while direct and indirect models, respectively, are considered by 
e.g.,   \cite{Zhao2000}, \cite{BelitserGhosal2003},
\cite{Castillo2008} and \cite{GaoZhou2014}, and,  e.g.,
\cite{KnapikVanderVaartVanZanten2011} and
\cite{KnapikSzaboVanderVaartVanZanten2014}. Bayesian procedures in the context of
slightly different Gaussian inverse problems and their asymptotic
properties are studied in, e.g., \cite{AgapiouLarssonStuart2013} and \cite{FlorensSimoni2010}.
However, our special attention is given to posterior consistency and optimal posterior concentration in
an oracle or minimax sense, which we
elaborate in the following.\\
In this paper we consider a sieve prior family $\{P_{\DiRvSo}\}_{\Di}$
where the prior distribution $P_{\DiRvSo}$ of the random parameter
sequence $\DiRvSo=(\DiRvSo_j)_{j\geq1}$ is Gaussian and degenerated
for all $j>m$. More precisely, the first $m$ coordinates
$\{\DiRvSo_j\}_{j=1}^{\Di}$ are independent and normally distributed
random variables while the remaining coordinates
$\{\DiRvSo_j\}_{j>\Di}$ are degenerated at a point. Note that the dimension
parameter $\Di$ plays the role of a tuning parameter. Assuming an
observation $\ObSo=(\ObSo_j)_{j\geq1}$ satisfying
$\ObSo_j=\DiRvSo_j+\sqrt{\ObSoNoL}\ObSoNo_j$, we denote by
$P_{\DiRvSo|\ObSo}$ the corresponding posterior distribution of
$\DiRvSo$ given $\ObSo$. Given a prior sub-family
$\{P_{\DiRvSo[\eDi]}\}_{\eDi}$ in dependence of the noise level
$\ObSoNoL$, our objective is the study of frequentist properties of the
associated posterior sub-family $\{P_{\DiRvSo[\eDi]|\ObSo}\}_{\eDi}$.
To be more precise, let $\TrSo$  be the realization of the random
parameter $\RvSo$ associated with the data-generating distribution and
denote by   $\Ex_{\TrSo}$ the corresponding expectation. A
quantity $\eRa$ which is up to a constant a lower and an upper bound of
the concentration of the posterior sub-family $\{P_{\DiRvSo[\eDi]|\ObSo}\}_{\eDi}$, i.e.,
\begin{equation}\label{in:de:ra}
\lim_{\ObSoNoL\to0}\Ex_{\TrSo}P_{\DiRvSo[\eDi]|\ObSo}((K)^{-1}\,
\eRa \leq \HnormV{\DiRvSo[\eDi]-\TrSo}^2\leq K\, \eRa)=1
\quad\text{with $1\leq  K<\infty$},\hfill
\end{equation}
 is called exact posterior
concentration (see, e.g., \cite{BarronBirgeMassart1999},
\cite{GhosalGhoshVanDerVaart2000} or \cite{Castillo2008} for a broader
discussion of the concept of posterior concentration). We shall
emphasise that the derivation of the posterior concentration relies strongly
on tail bounds for non-central $\chi^2$ distributions established
in \cite{Birge2001}. Moreover, if $\eRa\to0$ as $\ObSoNoL\to 0$ then the lower and
upper bound given in \eqref{in:de:ra} establish posterior consistency
and $\eRa$ is called exact posterior concentration rate.
Obviously, the exact  rate depends on the prior sub-family
$\{P_{\DiRvSo[\eDi]}\}_{\eDi}$ as well as on the
unknown parameter $\TrSo$. \\In the spirit of a frequentist oracle
approach, given a parameter $\TrSo$ we derive in this paper
a prior sub-family $\{P_{\DiRvSo[\treDi]}\}_{\treDi}$ with smallest
possible exact posterior concentration rate $\treRa$ which we
call, respectively, an oracle prior sub-family and an oracle posterior concentration rate. On the other hand side, following a minimax approach,  \cite{JohannesSchwarz2013},
for example, derive the minimax rate of convergence $\oeRa$ of the
maximal mean integrated squared error (MISE) over a given class $\cwSo$ of parameters  (introduced below).
We construct a sub-family $\{P_{\DiRvSo[\oeDi]}\}_{\oeDi}$ of prior
distributions with exact posterior concentration rate $\oeRa$
uniformly over  $\cwSo$ which does not depend on the true parameter $\TrSo$ but only on the set of possible
parameters  $\cwSo$. It is interesting to note that  in a direct GSSM
\cite{Castillo2008} establishes  up to a constant the minimax-rate as
an upper bound  of the posterior concentration, while
the derived lower bound features a logarithmic factor compared to the
minimax rate. \cite{ArbelGayraudRousseau2013}, for example,
in a direct GSSM and \cite{KnapikSzaboVanderVaartVanZanten2014} in an
indirect GSSM provide only upper bounds of the posterior
concentration rate which differ up to a logarithmic factor from the
minimax rate.
We shall emphasize, that the prior specifications we propose in this
paper lead to exact posterior concentration rates  that  are optimal
in an oracle or  minimax sense over certain classes of parameters not only in the direct model but also in the
 more general indirect model. However, both oracle and minimax sieve prior are unfeasible in
practise since they rely on the knowledge of either $\TrSo$ itself or
its smoothness.

Our main contribution in this paper is the construction of a
hierarchical prior $P_{\RvDiRvSo}$ that is adaptive. Meaning that, given a
 parameter $\TrSo\in\Hspace$ or a  classes $\cwSo\subset\Hspace$ of
parameters, the posterior distribution $P_{\RvDiRvSo|\ObSo}$ contracts,
respectively, at the oracle rate or the minimax rate over $\cwSo$
while the hierarchical prior $P_{\RvDiRvSo}$ does not rely neither on the
knowledge of $\TrSo$ nor the class $\cwSo$.  Let us briefly elaborate on
the hierarchical structure of the prior which induces an additional prior on the
tuning parameter $\Di$, i.e., $\Di$ itself is a realisation of a
random variable $\RvDi$. We construct a prior for $\RvDi$ such
  that the marginal posterior for $\RvDiRvSo$ (obtained by integrating out
  $\RvDi$ with respect to its posterior) contracts exactly at the
  oracle concentration rate. This is possible for every $\TrSo$ whose
  components differ from the components of the prior mean  infinitely many times. In addition, for
  every $\TrSo$ in the class $\cwSo$ we show that the posterior
  distribution $P_{\RvDiRvSo|\ObSo}$ contracts at least at the minimax
  rate $\oeRa$ and that the  corresponding Bayes estimate is
  minimax-optimal. Thereby, the proposed  Bayesian procedure is \emph{minimax
    adaptive} over the class $\cwSo$.

Although adaptation has attracted remarkable interest in the frequentist
literature, only few contributions are available in the Bayesian
literature on Gaussian sequence space models.  In a direct model
\cite{BelitserGhosal2003}, \cite{SzaboVanderVaartVanZanten2013},
\cite{ArbelGayraudRousseau2013} and \cite{GaoZhou2014} derive Bayesian
methods that achieve minimax adaptation while in an indirect Gaussian
sequence space model, to the best of our knowledge, only
\cite{KnapikSzaboVanderVaartVanZanten2014} has derived an adaptive
Bayesian procedure. In this paper, we extend previous results on
adaptation obtained through sieve priors to the indirect Gaussian
sequence space model. This requires a specification of the prior on
the tuning parameter $\RvDi$ different from the one used by, e.g.,
\cite{Zhao2000} and \cite{ArbelGayraudRousseau2013}. Interestingly,
our novel prior specification on $\RvDi$ improves the general results
of \cite{ArbelGayraudRousseau2013} since it allows to obtain
adaptation without a rate loss (given by a logarithmic factor) even in
the direct model. Compared to
\cite{KnapikSzaboVanderVaartVanZanten2014} our procedure relies on a
sieve prior while they use a family of Gaussian prior for $\RvSo$
that is not degenerate in any component of $\RvSo$ and where the
hyper-parameter is represented by the smoothness of the prior
variance. Their procedure is minimax-adaptive up to a logarithmic
deterioration of the minimax rate  on certain smoothness classes
for $\TrSo$ which is, instead, avoided by our
  procedure.

The rest of the paper is organised as follows. The prior scheme is
specified in Section \ref{s:mo}. In Section \ref{s:tr} we derive the
lower and upper bound of the posterior concentration, the oracle
posterior concentration rate and the minimax rate. In Section
\ref{s:ad} we introduce  a  prior distribution $P_{\RvDi}$ for the
random dimension  $\RvDi$ and we prove adaptation of the hierarchical Bayes procedure. 
The proofs are given in the appendix.\\


%% file: _2Model.tex
\section{Basic model assumptions}\label{s:mo}
Let us consider a Gaussian {prior} distribution for the   parameter
$\RvSo=(\RvSo_j)_{j\geq1}$, that is, $\{\RvSo_j\}_{j\geq1}$ are independent, normally distributed  with prior means
$(\RvSoEr_j)_{j\geq1}$ and prior variances $(\RvSoVa_j)_{j\geq1}$. 
Standard calculus  shows that the posterior distribution of  $\RvSo$
 given  $\ObSo=(\ObSo_j)_{j\geq1}$ is Gaussian, that is, given $\ObSo$,
$\{\RvSo_j\}_{j\geq1}$   are conditionally independent, normally distributed random variables  with
posterior variance
$\sigma_j:=\Var(\RvSo_j\vert\ObSo)=(\Ev_j^2\ObSoNoL^{-1}+\RvSoVa^{-1}_j)^{-1}$
and posterior mean
$ \So^{\ObSo}_j:=\Ex[\RvSo_j\vert\ObSo]=\sigma_j({\RvSoVa^{-1}_j\RvSoEr_j+\Ev_j\ObSoNoL^{-1}\ObSo_j})$, for all $j\in\Nz$.
Taking this as a starting point, we construct a sequence of hierarchical Sieve prior distributions. To be more precise,
let us denote by $\dirac_x$ the Dirac measure in the point $x$.  Given $\Di\in\Nz$, we
consider the  independent random variables $\set{\DiRvSo_j}_{j\geq1}$
with marginal distributions
\begin{equation}\label{mo:de:DiRvSo}
\DiRvSo_j\sim
\cN(\RvSoEr_j,\RvSoVa_j),\; 1\leq j\leq \Di \mbox{ and } \DiRvSo_j\sim
\dirac_{\RvSoEr_j},\; \Di<j,
\end{equation}
resulting in the degenerate  prior distribution
$P_{\DiRvSo}$. Here, we use the notation $\DiRvSo=\suite{\DiRvSo}$.
Consequently,  $\{\DiRvSo_j\}_{j\geq1}$ are conditionally independent given $\ObSo$ and their { posterior}
distribution $P_{\DiRvSo_j|\ObSo}$ is Gaussian with mean $\So^{\ObSo}_j$ and variance $\sigma_j$ for
$1\leq j\leq \Di$ while being degenerate on $\RvSoEr_j$ for $j>\Di$.\\
Let $\Ind{A}$ denote the
  indicator function which takes the value one if the condition $A$ holds true, and the value zero otherwise.
We consider the posterior mean $\hSo^\Di=\suite{\hSo^\Di}:=\Ex[\DiRvSo\vert \ObSo]$ given  for $j\geq1$ by
$\hSo_j^{\Di}:=\So^{\ObSo}_j \Ind{\set{j\leq \Di}}+
\RvSoEr_j \Ind{\set{j> \Di}}$ as Bayes estimator of $\So$. We shall
emphasize an improper specification of the prior, that is, $\RvSoEr=\suite{\RvSoEr}\equiv 0$ and $\RvSoVa=\suite{\RvSoVa}\equiv
\infty$. Obviously, in this situation
$\So^{\ObSo}=Y/\Ev=\suite{Y_j/\Ev}$ and $\PoVa=\ObSoNoL/\Ev^2=\suite{\ObSoNoL/\Ev^2}$ are the posterior mean and
variance sequences, respectively. Consequently, under the improper prior
specification, for each $\Di\in\Nz$ the posterior mean
$\hSo^\Di=\Ex[\RvSo^\Di|\ObSo]$ of $\RvSo^\Di$ corresponds to an
orthogonal projection estimator, i.e.,
$\hSo^\Di=(Y/\Ev)^m$ with   $(Y/\Ev)^m_j=Y_j/\Ev_j\Ind{\set{1\leq j\leq \Di}}$.
\\
From a  Bayesian point of view the thresholding parameter $\Di$ is a hyper-parameter and hence, we may complete
the prior specification by introducing a prior distribution on it. Consider a random
thresholding parameter $\RvDi$ taking its values in $\set{1,\dotsc,\DiMa}$ for some  $\DiMa\in\Nz$   with
prior distribution $P_{\RvDi}$. Both
$\DiMa$
and $P_{\RvDi}$ will be specified in Section
\ref{s:ad}. 
Moreover, the distribution of the random variables $\set{\ObSo_j}_{j\geq1}$ and $\{\RvDiRvSo_j\}_{j\geq1}$ conditionally on $\RvDi$ are determined by
\begin{equation*}
\ObSo_j=\Ev_j\RvDiRvSo+\sqrt{\ObSoNoL}\ObSoNo_j\quad\text{ and }\quad
\RvDiRvSo_j=\RvSoEr_j+\sqrt{\RvSoVa}_j \RvSoNo_j \Ind{\set{1\leq j\leq \RvDi}}
\end{equation*}
where $\{\ObSoNo_j,\RvSoNo_j\}_{j\geq1}$
are iid. standard normal random variables independent of
$\RvDi$. Furthermore, the posterior mean $\hSo:=\Ex[\RvDiRvSo|\ObSo]$ satisfies
$\hSo_j=\RvSoEr_j$ for $j>\DiMa$ and $\hSo_j=\RvSoEr_j\,P(1\leq \RvDi<
j\vert\ObSo) + \So_j^{\ObSo}\,P(j\leq \RvDi\leq \DiMa\vert\ObSo)$ for
all $1\leq j\leq \DiMa$. It is important to note, that the marginal posterior
distribution $P_{\RvDiRvSo|\ObSo}$ of $\RvDiRvSo=\suite{\RvDiRvSo}$
given the observation $\ObSo$ does depend on the prior specification
and the observation only, and hence it is fully data-driven. Revisiting
the improper prior specification introduced above,
the data-driven Bayes estimator equals a shrunk orthogonal
projection estimator. More precisely, we have  $\hSo_j=  P(j\leq
  \RvDi\leq \DiMa\vert\ObSo)\times Y_j/\Ev_j\Ind{\set{1\leq j\leq
      \DiMa}}$. Interestingly, rather than using the data  to select the dimension
  parameter $\Di$ in the set of possible values $\{1,\dotsc,\DiMa\}$,
  the Bayes estimator uses all components, up to $\DiMa$, shrunk by a
  weight decreasing with the index.

%% file: _3Rate.tex
\section{Optimal concentration rate}\label{s:tr}
\subsection{Consistency}
Note that conditional on $\ObSo$ the random variables $\{\DiRvSo_j-\TrSo[j]\}_{j=1}^m$ are independent and
normally distributed with conditional mean $\So^{\ObSo}_j-\TrSo[j]$ and conditional variance
$\sigma_j$.
The next assertion
presents   a version of tail bounds for sums of independent squared
Gaussian random variables. It is shown in the appendix using a result
due to \cite{Birge2001} which can be shown along  the lines of the proof of Lemma 1 in \cite{LaurentLM2012}.
\begin{lem}\label{tr:le:te}
  Let $\{X_j\}_{j\geq1}$ be  independent and normally distributed r.v.\ with mean
  $\alpha_j\in\Rz$ and standard deviation $\beta_j\geq0$, $j\in\Nz$. For  $m\in\Nz$ set
  $S_{m}:=\sum_{j=1}^{m}X_j^2$ and consider  $v_m\geq \sum_{j=1}^m\beta_j^2$, $t_m\geq \max_{1\leq
    j\leq m}\beta_j^2$ and $r_m\geq\sum_{j=1}^m\alpha_j^2$.  Then for all $c\geq 0$ we have
  \begin{align}\label{tr:le:te:e1}
    &\sup_{m\geq 1}\exp\Bigl({\frac{c(c\wedge1)(v_m+ 2r_m)}{4 t_m}}\Bigr)P\big(S_m-\Ex S_m\leq -c(v_m+2r_m)\big) \leq 1;\\ \label{tr:le:te:e2}
&\sup_{m\geq 1}\exp\Bigl({\frac{c(c\wedge1)(v_m + 2r_m)}{4t_m}}\Bigr) P\big(S_m-\Ex S_m \geq \frac{3c}{2}(v_m +
    2r_m)\big)\leq1.
  \end{align}
\end{lem}
\noindent A major step towards establishing a concentration rate of the
{posterior} distribution consists in finding a finite sample bound
for a fixed $\Di\in\Nz$. We express these bounds in terms of
\begin{gather*}
\gb_\Di:=\sum_{j>\Di}(\TrSo[j]-\RvSoEr_j)^2,\quad\oPoVa:=\sum_{j=1}^{\Di}\sigma_j \quad\mbox{with }\sigma_j=(\Ev_j^2\ObSoNoL^{-1}+\RvSoVa_j^{-1})^{-1};\\
\mPoVa:=\max_{1\leq j\leq \Di}\sigma_j\quad\mbox{and}\quad \gr_\Di:= \sum_{j=1}^{m}(\Ex_{\TrSo}[\So^{\ObSo}_j]-\TrSo[j])^2=\sum_{j=1}^{m}\sigma_j^2\RvSoVa^{-2}_j(\RvSoEr_j-\TrSo[j])^2.
\end{gather*}
\begin{prop}\label{tr:pr:pbm}For all $\Di\in\Nz$, for all $\ObSoNoL>0$
  and for all $ 0<c< 1/5$ we
  have
\begin{align}\label{tr:pr:pbm:e1}
& \Ex_{\TrSo}P_{\DiRvSo|\ObSo}(
\HnormV{\DiRvSo-\TrSo}^2> \gb_\Di+ 3\oPoVa + {3}\,\Di\,\mPoVa/2+4\gr_\Di)\leq 2\exp(-{\Di}/{36});\\\label{tr:pr:pbm:e2}
& \Ex_{\TrSo}P_{\DiRvSo|\ObSo}( \HnormV{\DiRvSo-\TrSo}^2< \gb_\Di+\oPoVa- 4\, c\, (\Di\,\mPoVa+\gr_\Di))\leq 2\exp(-c^2\Di/2).\hfill
\end{align}
\end{prop}
\noindent 
The desired convergence to zero of all the aforementioned sequences
necessitates to consider an appropriate sub-family
$\{P_{\RvSo^{\eDi}}\}_{\eDi}$  in dependence of the noise level
$\ObSoNoL$, notably introducing consequently sub-sequences $(\oPoVa[\eDi])_{\eDi\geq1},(\mPoVa[\eDi])_{\eDi\geq1}$ and
 $(\gr_{\eDi})_{\eDi\geq1}$. 
\begin{ass}\label{tr:as:pr}
There  exist constants $
0<\ObSoNoL_{\TrSy}:=\ObSoNoL_{\TrSy}{(\TrSo,\Ev,\RvSoEr,\RvSoVa)}<1 $ and   $1\leq
K:=K{(\TrSo,\Ev,\RvSoEr,\RvSoVa)}<\infty$  such
that
the Sieve sub-family
$\{P_{\RvSo^{\eDi}}\}_{\eDi}$ of  prior
 distributions satisfies the condition
 $\sup_{0<\ObSoNoL<\ObSoNoL_{\TrSy}}(\gr_{\eDi}\vee \eDi\mPoVa[\eDi])/(\gb_{\eDi}\vee\oPoVa[\eDi])\leq K$.
\end{ass}%
\noindent The following corollary can be immediately deduced from
Proposition \ref{tr:pr:pbm} and we omit its proof.%
\begin{coro}\label{tr:co:pbm} Under Assumption \ref{tr:as:pr}
  for all $0<\ObSoNoL<\ObSoNoL_{\TrSy}$ and $0<c<1/(8K)$ hold
\begin{align}\label{tr:co:pbm:e1}
& \Ex_{\TrSo}P_{\DiRvSo[\eDi]|\ObSo}(
\HnormV{\DiRvSo[\eDi]-\TrSo}^2> (4+({11}/{2})K)[\gb_{\eDi}\vee\oPoVa[\eDi]])\leq 2\exp(-\frac{\eDi}{36});\\\label{tr:co:pbm:e2}
& \Ex_{\TrSo}P_{\DiRvSo[\eDi]|\ObSo}( \HnormV{\DiRvSo[\eDi]-\TrSo}^2<(1-8\,c\,K)[\gb_{\eDi}\vee\oPoVa[\eDi]])\leq 2\exp(-c^2\eDi/2).\hfill
\end{align}
\end{coro}
\noindent Note that the sequence  $(
\gb_{\eDi}\vee\oPoVa[\eDi])_{\eDi\geq 1}$ generally does not converge
to zero. However, supposing that $\eDi\to\infty$ as $\ObSoNoL\to0$
then it follows from the dominated convergence
theorem that $\gb_{\eDi}=o(1)$. Hence, assuming additionally that $\oPoVa[\eDi]=o(1)$
holds true is sufficient to ensure that  $(	\gb_{\eDi}\vee\oPoVa[\eDi])_{\eDi\geq 1}$
converges to zero and it is indeed
a posterior concentration rate. The next assertion summarises this
result and we omit its elementary proof.
\begin{prop}[Posterior consistency]\label{tr:th:pc} Let Assumption
  \ref{tr:as:pr} be satisfied. If
  $\eDi\to\infty$ and $\oPoVa[\eDi]=o(1)$  as $\ObSoNoL\to0$, then
  \begin{equation*}
\lim_{\ObSoNoL\to0}\Ex_{\TrSo}P_{\DiRvSo[\eDi]|\ObSo}\big(  (10K)^{-1}[\gb_{\eDi}\vee\oPoVa[\eDi]]\leq\HnormV{\DiRvSo[\eDi]-\TrSo}^2\leq 10K[\gb_{\eDi}\vee\oPoVa[\eDi]]\big)=1.
\end{equation*}
\end{prop}
\noindent The last assertion shows that $(\gb_{\eDi}\vee\oPoVa[\eDi])_{\eDi\geq1}$ is up to a constant a lower and upper
bound of the concentration rate associated with the Sieve sub-family
$\{P_{\RvSo^{\eDi}}\}_{\eDi}$ of prior distributions.  It is easily
shown that it also provides an upper bound of the frequentist risk of
the associated Bayes estimator.
\begin{prop}[Bayes estimator consistency]\label{tr:pr:bc} Let the assumptions of Proposition \ref{tr:th:pc} be
  satisfied. Consider the Bayes estimator $\hSo^{\eDi}:=\Ex[\DiRvSo[\eDi]|\ObSo]$ then
  \begin{equation*}
\Ex_{\TrSo} \HnormV{\hSo^{\eDi}-\TrSo}^2 \leq (2+K)[\gb_{\eDi}\vee\oPoVa[\eDi]]
\end{equation*}
and consequently $\Ex_{\TrSo} \HnormV{\hSo^{\eDi}-\TrSo}^2=o(1)$ as $\ObSoNoL\to0$.
\end{prop}
\noindent The previous results are obtained under Assumption \ref{tr:as:pr}. However, it may be difficult to verify whether  a
given sub-family of priors  $\{P_{\RvSo^{\eDi}}\}_{\eDi}$  satisfies such an assumption. 
Therefore, we now introduce an assumption which states a more
precise requirement on the prior variance and that can be more easily
verified. Define for $j,m\in\Nz$
\begin{equation*}
\Evs_j:=\Ev_j^{-2},\quad \mEvs:=\max_{1\leq j\leq
  \Di}\Evs_{j},\quad\oEvs:=m^{-1}\sum_{j=1}^\Di\Evs_{j}\quad\text{ and
}\quad \eRa^{\Di}:=[\gb_{\Di}\vee \ObSoNoL\,\Di\,\oEvs[\Di]].
\end{equation*}
\begin{ass}\label{tr:as:pv} Let $\DiMa:=\max\{1\leq m\leq \gauss{\ObSoNoL^{-1}}:\ObSoNoL\mEvs\leq \Evs_1\}$.
There exists a finite constant $d>0$ such that
$\RvSoVa_j\geq d[\ObSoNoL^{1/2}\Evs_j^{1/2} \vee\ObSoNoL\Evs_j]$ for all $1\leq j \leq \DiMa$ and for all $\ObSoNoL\in(0,1)$.
\end{ass}
\noindent Note that in the last Assumption the defining set of $\DiMa$ is not empty, since
$\ObSoNoL \mEvs[1]\leq\Evs_1$ for all $\ObSoNoL\leq 1$.
Moreover,  under Assumption \ref{tr:as:pv},  by some elementary algebra, it is readily verified
 for all $1\leq j \leq \DiMa$ that
\begin{equation*}
  1 \leq \ObSoNoL\Evs_j/\PoVa[j]\leq (1+1/d) \quad\text{
    and } \quad\PoVa[j]/\RvSoVa_j \leq (1\wedge d^{-1}\ObSoNoL^{1/2} \Evs_j^{1/2})
\end{equation*}
which in turn implies  for all $1\leq \Di \leq \DiMa$ that
\begin{equation*}
\gr_{\Di}\leq
d^{-2}\HnormV{\RvSoEr-\TrSo}^2\ObSoNoL\,\mEvs[\Di],\;  1\leq  \ObSoNoL\,
\Di\,\mEvs[\Di](\Di\mPoVa)^{-1}\;\text{and}\; 1\leq\ObSoNoL \,\Di \oEvs[\Di]\,(\oPoVa[\Di])^{-1} \leq(1+1/d).
\end{equation*}
We will use these elementary bounds in the sequel without further reference.
Returning  to the Sieve sub-family $\{P_{\RvSo^{\eDi}}\}_{\eDi}$ of prior distributions, if in addition
to Assumption \ref{tr:as:pv}  there exists a  constant
$1\leq L:=L(\TrSo,\Ev,\RvSoEr)<\infty$ such that
\begin{equation}\label{tr:as:pv:ac}
\sup_{0<\ObSoNoL<1}
\ObSoNoL\,\eDi\,\mEvs[\eDi](\eRa^{\eDi})^{-1}\leq L
\end{equation}
and $\eRa^{\eDi}=o(1)$ as $\ObSoNoL\to0$ hold true, then the sub-family $\{P_{\RvSo^{\eDi}}\}_{\eDi}$ satisfies  Assumption \ref{tr:as:pr}  with $K:=((1+d^{-1})\vee d^{-2}\HnormV{\TrSo-\RvSoEr}^2)L$. Indeed, if $\eRa^{\eDi}=o(1)$ and, hence $\eRa[{\ObSoNoL}]^{\Di_{\ObSoNoL}}\leq\Evs_1/L$ for all
  $\ObSoNoL\in(0,\ObSoNoL_{\TrSy})$, then $\eDi\leq\DiMa$ holds true for all
  $\ObSoNoL\in(0,\ObSoNoL_{\TrSy})$  since
  $\ObSoNoL\eDi\Evs_1\leq\ObSoNoL\eDi\mEvs[\eDi]\leq L
  \eRa[{\ObSoNoL}]^{\Di_{\ObSoNoL}}\leq \Evs_1$ and thus  $\eDi\leq
  \gauss{\ObSoNoL^{-1}}$  and  $\ObSoNoL \mEvs[\eDi]\leq
  \Evs_1$. 
  In other words, for all $\ObSoNoL\in(0,\ObSoNoL_{\TrSy})$ we can
  apply Assumption \ref{tr:as:pv} and  the claim follows  taking into account
  the aforementioned  elementary bounds. Note further that the constant $K$ does not depend  on the prior variances
$\RvSoVa$ but only on the constant $d$ given by Assumption
\ref{tr:as:pv}. The  next assertion follows immediately from Corollary
    \ref{tr:co:pbm} and we omit its proof.
\begin{coro}\label{tr:co:pb}Under  Assumption \ref{tr:as:pv} consider a sub-family $\{P_{\RvSo^{\eDi}}\}_{\eDi}$  such
  that \eqref{tr:as:pv:ac} and $\eRa^{\eDi}=o(1)$ as $\ObSoNoL\to0$
  are satisfied, then there exists $\ObSoNoL_{\TrSy}\in(0,1)$ such that  for
  all $0<\ObSoNoL<\ObSoNoL_{\TrSy}$ and $0<c<1/(8K)$ with $K=((1+d^{-1})\vee
  d^{-2}\HnormV{\TrSo-\RvSoEr}^2)L$ hold
\begin{align}\label{tr:co:pb:e1}
& \Ex_{\TrSo}P_{\DiRvSo[\eDi]|\ObSo}\big(
\HnormV{\DiRvSo[\eDi]-\TrSo}^2> (4+({11}/{2})K)\eRa^{\eDi}\big)\leq 2\exp(-\frac{\eDi}{36});\\\label{tr:co:pb:e2}
& \Ex_{\TrSo}P_{\DiRvSo[\eDi]|\ObSo}\big( \HnormV{\DiRvSo[\eDi]-\TrSo}^2<(1-8\,c\,K)(1+d^{-1})^{-1}\eRa^{\eDi}\big)\leq 2\exp(-c^2\eDi/2).\hfill
\end{align}
\end{coro}
\noindent The result implies consistency if $\eDi\to\infty$ as $\ObSoNoL\to0$ but it does not answer the question
of an optimal rate in a satisfactory way.
\subsection{Oracle concentration rate}
Considering the Sieve family $\{P_{\RvSo^{\Di}}\}_\Di$ of prior
distributions, the sequence $(\eRa^{\eDi})_{\eDi\geq1}$ provides up to constants a lower and upper
bound for the posterior concentration rate for each sub-family
$\{P_{\RvSo^{\eDi}}\}_{\eDi}$ satisfying the conditions of Corollary \ref{tr:co:pb}. Observe that the term $\gb_{\eDi}$ and hence the rate
depends on the parameter of interest $\TrSo$. Let us minimise  the rate for each
$\TrSo$ separately. For  a sequence $(a_m)_{m\geq1}$ with minimal
value in $A$ we set $\argmin\nolimits_{m\in A}\set{a_m}:=\min\set{m:a_m\leq
  a_{k},\forall k\in A}$ and define for all  $\ObSoNoL>0$
\begin{multline}\label{tr:de:tre}
  \treDi:=\treDi(\TrSo,\RvSoEr,\Ev):=\argmin_{\Di\geq 1} \set{\eRa^{\Di}}\text{ and }\\
  \treRa:=\treRa(\TrSo,\RvSoEr,\Ev):=\eRa^{\treDi}=\min_{\Di\geq 1}\eRa^{\Di}\quad .
\end{multline}
We may emphasise that $\treRa=o(1)$ as $\ObSoNoL\to 0$. Indeed, for
all $\delta>0$ there exists a dimension $\Di_\delta$ and a noise level $\ObSoNoL_\delta$ such that $\treRa\leq[\gb_{\Di_\delta}\vee\ObSoNoL_{\delta}\, \Di_\delta\,
  \oEvs[\Di_\delta]]\leq \delta$ for all
  $0<\ObSoNoL\leq\ObSoNoL_\delta$.
Obviously, given $\TrSo\in\cSo$ the rate $\treRa$ is a lower bound for all posterior concentration
rates $\eRa^{\eDi}$ associated with a
prior sub-family $\{P_{\RvSo^{\eDi}}\}_{\eDi}$ satisfying the conditions of
Corollary \ref{tr:co:pb}.
Moreover, the  next assertion establishes $\treRa$ up to constants as upper and  lower bound for
the concentration rate associated with the sub-family
$\{P_{\RvSo^{\treDi}}\}_{\treDi}$. Consequently,   $\treRa$ is called oracle posterior concentration rate
and $\{P_{\RvSo^{\treDi}}\}_{\treDi}$ oracle prior sub-family. The assertion follows again from
Corollary  \ref{tr:co:pbm} (with $c=1/(9K)$) and we omit its proof.
\begin{theo}[Oracle posterior concentration rate]\label{tr:pr:ora}
Suppose that Assumption \ref{tr:as:pv} holds true and that there
exists a constant  $1\leq L^{\TrSy}:=L^{\TrSy}(\TrSo,\Ev,\RvSoEr)<\infty$ such that
\begin{equation}\label{tr:pr:ora:ac}
\sup_{0<\ObSoNoL<1} \ObSoNoL\,\treDi\,\mEvs[\treDi](\treRa)^{-1}\leq L^{\TrSy}.
\end{equation}
If in addition $\treDi\to\infty$ as $\ObSoNoL\to0$ and $K^{\TrSy}:=10((1+d^{-1})\vee d^{-2}\HnormV{\TrSo-\RvSoEr}^2)L^{\TrSy}$, then
\[\lim_{\ObSoNoL\to0}\Ex_{\TrSo}P_{\DiRvSo[\treDi]|\ObSo}(
 (K^{\TrSy})^{-1} \treRa\leq \HnormV{\DiRvSo[\treDi]-\TrSo}^2\leq K^{\TrSy} \treRa)=1.\]
\end{theo}
\noindent Note that $\treDi\to\infty$ as $\ObSoNoL\to 0$ if and only if
  $\gb_{\Di}>0$ for all $m\geq1$. Roughly speaking,  the last assertion
  establishes $\treRa$ as oracle posterior concentration rate for all
  parameter of interest $\TrSo$ with components differing from the
  components of the prior mean
  $\RvSoEr$ infinitely many times.
 However, we do not need this additional assumption to prove the next
 assertion which establishes  $\treRa$ as  oracle
 rate for the family $\{\hSo^{\Di}\}_{\Di}$ of Bayes estimator
and that $\hSo^{\treDi}$ is  an oracle Bayes estimator.
  \begin{theo}[Oracle Bayes estimator]\label{tr:pr:bo} Consider the family $\{\hSo^{\Di}\}_{\Di}$ of Bayes
estimators. Under  Assumption
    \ref{tr:as:pv} we have \begin{inparaenum}[\dr\upshape(i\upshape)]
\item\label{tr:pr:bo:i}$\Ex_{\TrSo}
  \HnormV{\hSo^{\treDi}-\TrSo}^2 \leq (2+d^{-2}\HnormV{\TrSo-\RvSoEr}^2)\treRa$ and
\item\label{tr:pr:bo:ii}$\inf_{\Di\geq 1}  \Ex_{\TrSo}
  \HnormV{\hSo^{\Di}-\TrSo}^2\geq (1+1/d)^{-2}\treRa$
\end{inparaenum}  for all  $\ObSoNoL\in(0,\ObSoNoL_o)$. \end{theo}
\noindent Note that, the oracle choice  $\treDi$ depends  on the parameter of interest $\TrSo$ and thus the
oracle   Bayes estimator $\hSo^{\treDi}$ as well as the associated  oracle
sub-family $\{P_{\RvSo^{\treDi}}\}_{\treDi}$ of prior distributions
 are generally not feasible.
\subsection{Minimax concentration rate}
In the spirit of a minimax theory
 we are interested in the following   in a uniform rate over a class
 of parameters rather than optimising the rate for each $\TrSo$
 separately. Given a strictly positive  and  non-increasing sequence
 $\wSo=\suite{\wSo}$ with  $\wSo_1=1$ and $\lim_{j\to\infty}\wSo_j=0$ consider for  $\theta\in\Hspace$  its weighted norm
$\normV{\theta}_\wSo^2:=\sum_{j\geq1}\theta_j^2/\wSo_j$. We define $\wHspace$ as the completion of $\Hspace$ with
respect to $\norm_\wSo$.
In order to formulate the optimality of the posterior concentration rate let us
define
\begin{multline}\label{tr:de:oe}
  \oeDi:=\oeDi(\wSo,\Ev):=\argmin_{\Di\geq 1} \set{\wSo_{\Di}\vee \ObSoNoL\, \Di\, \oEvs}\text{ and }\\
  \oeRa:=\oeRa(\wSo,\Ev):=[\wSo_{\oeDi}\vee \ObSoNoL\, \oeDi\, \oEvs[\oeDi]]\quad \text{for all }\ObSoNoL>0.
\end{multline}
We remark that  $\oeRa=o(1)$ and  $\oeDi\to\infty$ as $\ObSoNoL\to 0$
since $\wSo$ is strictly positive and tends monotonically to
zero. We assume in the following that the
parameter $\TrSo$ belongs to the ellipsoid $\cwrSo:=\set{\theta\in \wHspace:
  \normV{\theta-\RvSoEr}_\wSo^2\leq\rSo}$ and therefore, $\gb_\Di(\TrSo)\leq
\ga_{\Di}\rSo$. Note that $\treRa=\min_{\Di\geq1}[\gb_{\Di}\vee
\ObSoNoL\, \Di\, \oEvs]\leq (1\vee\rSo)\min_{\Di\geq1}[\wSo_{\Di}\vee \ObSoNoL\,
\Di\, \oEvs]=(1\vee\rSo)\oeRa$ and $\HnormV{\TrSo-\RvSoEr}^2\leq\rSo$, and hence from Theorem \ref{tr:pr:bo} it follows $\Ex_{\TrSo}
  \HnormV{\hSo^{\treDi}-\TrSo}^2 \leq (2+\rSo/d^2)(1\vee\rSo)\oeRa$. On the
  other hand side, given an estimator $\hSo$ of $\So$ let $\sup_{\So\in\cwrSo}\Ex_\So\HnormV{\hSo-\So}^2$
denote the maximal mean integrated squared error over the class
$\cwrSo$. It has been shown in \cite{JohannesSchwarz2013} that $\oeRa$
 provides up to a constant a lower bound for the maximal  MISE over
 the class
 $\cwrSo$ (assuming a prior mean  $\RvSoEr=0$) if the next assumption is satisfied.
\begin{ass}\label{tr:as:mi}
Let $\wSo$ and $\Ev$ be sequences such that
\begin{equation}
0<\kappa^{\OpSy}:=\kappa^{\OpSy}(\wSo,\Ev):=\inf_{0<\ObSoNoL<\ObSoNoL_o}\set{(\oeRa)^{-1}[\wSo_{\oeDi}\wedge \ObSoNoL\,
\oeDi\, \oEvs[\oeDi]]}\leq 1.\label{tr:as:mi:e1}
\end{equation}
\end{ass}
\noindent We may emphasise that under Assumption \ref{tr:as:mi} the rate
$\oeRa=\oeRa(\wSo,\Ev)$ is optimal in a minimax sense and the Bayes
estimate $\hSo^{\treDi}$ attains the minimax rate up to a constant.
However, the dimension parameter $\treDi$ depends still on the
parameter of interest $\TrSo$. Therefore, let us consider the Bayes
estimate $\hSo^{\oeDi}$ and the sub-family
$\{P_{\RvSo^{\oeDi}}\}_{\oeDi}$ of prior distributions
which do not depend anymore on the parameter
of interest $\TrSo$ but only on the set of possible parameters $\cwrSo$
characterised by the weight sequence $\wSo$. The next assertion can be shown 
along the lines of the proof of Theorem \ref{tr:pr:bo},
and, hence we omit its proof.
\begin{theo}[Minimax optimal Bayes estimator]\label{tr:th:obe} Let Assumption
  \ref{tr:as:pv} be satisfied. Considering the Bayes estimator
  $\hSo^{\oeDi}:=\Ex[\DiRvSo[\oeDi]|\ObSo]$  we have
  \begin{equation*}
\sup_{\TrSo\in\cwrSo}\Ex_{\TrSo}
  \HnormV{\hSo^{\oeDi}-\TrSo}^2 \leq (2+\rSo/d^2)(1\vee\rSo)\oeRa\quad\text{for all }\ObSoNoL\in(0,\ObSoNoL_o).
\end{equation*}
\end{theo}
\noindent The last assertion establishes the minimax optimality of the Bayes
estimate $\hSo^{\oeDi}$ over the class $\cwrSo$.  Moreover,
 the minimax rate $\oeRa$  provides up to a constant a lower and an
 upper bound for the posterior concentration rate associated with the
 prior sub-family $\{P_{\RvSo^{\oeDi}}\}_{\oeDi}$, which
 is summarised in the next assertion.
\begin{theo}[Minimax optimal posterior concentration rate]\label{tr:th:opc}
Let Assumption \ref{tr:as:pv} and \ref{tr:as:mi} hold
true. If there
exists a constant $1\leq L^{\OpSy}:=L^{\OpSy}(\wSo,\Ev)<\infty$ such that
\begin{equation}\label{tr:as:mi:ac}
\sup_{0<\ObSoNoL<\ObSoNoL_o}\ObSoNoL\,\oeDi\,\mEvs[\oeDi](\oeRa)^{-1}\leq L^{\OpSy}
\end{equation}
and  $K^{\OpSy}:=K^{\OpSy}(\rSo,\wSo,\Ev, d, \kappa):=10 ((1+1/d)\vee \rSo/d^{2})(1\vee\rSo)(L^{\OpSy}/\kappa^{\OpSy})$, then
\[\lim_{\ObSoNoL\to0}\inf_{\TrSo\in\cwrSo}\Ex_{\TrSo}P_{\DiRvSo[\oeDi]|\ObSo}(
(K^{\OpSy})^{-1} \oeRa\leq \HnormV{\DiRvSo[\oeDi]-\TrSo}^2\leq K^{\OpSy} \oeRa)=1.\]
\end{theo}
\noindent Comparing the last result with the result of Theorem \ref{tr:pr:ora} and keeping
in mind that $(1\vee \rSo)\oeRa\geq\treRa$, the posterior
concentration rate associated with the prior sub-family
$\{P_{\RvSo^{\oeDi}}\}_{\oeDi}$ is of order of the minimax rate $\oeRa$ uniformly for all parameter of
interest $\TrSo\in\cwrSo$. However, for certain parameter $\TrSo$ the minimax rate $\oeRa$ may be far slower than
the oracle rate $\treRa$. For example, as shown in case
{\bf\small[P-P]} in the following illustration the
minimax rate $\oeRa$ is of order $O(\ObSoNoL^{2p/(2a+2p+1)}$
) while it is not hard to see, that for all parameter $\TrSo$ with $\gb_\Di\asymp \exp (-\Di^{2p})$  the
oracle rate is of order $O(\ObSoNoL|\log\ObSoNoL|^{(2a+1)/(2p)})$ (see
case {\bf\small[E-P]}). Moreover, the optimal choice $\oeDi$ of the
dimension parameter still depends on the class $\cwrSo$, which might
be unknown in practise, therefore we will consider in the next section
a fully data-driven choice using a hierarchical specification of the  prior distribution.

\begin{illu}\label{tr:il:mi} We illustrate the last assumptions and the minimax rate for typical choices of the sequences $\wSo$ and $\Ev$. For two strictly
positive sequences $(a_j)_{j\geq1}$ and $(b_j)_{j\geq1}$  we  write $a_j\asymp b_j$, if $(a_j/b_j)_{j\geq 1}$ is bounded away from $0$ and infinity.
\begin{itemize}
\item[{\bf\small[P-P]}] Consider $\wSo_j\asymp j^{-2p}$ and $\Ev_j^2\asymp
  j^{-2a}$ with $p> 0$ and $a>0$ then $\oeDi\asymp
  \ObSoNoL^{-1/(2p+2a+1)}$ and $\oeRa\asymp\ObSoNoL^{2p/(2a+2p+1)}$.
\item[{\bf\small[E-P]}] Consider $\wSo_j\asymp\exp (-j^{2p}+1)$ and
  $\Ev_j^2\asymp j^{-2a}$ with $p> 0$ and $a>0$ then $\oeDi\asymp |\log
  \ObSoNoL -\frac{2a+1}{2p}(\log|\log \ObSoNoL|)|^{1/(2p)}$ and
  $\oeRa\asymp \ObSoNoL |\log \ObSoNoL|^{(2a+1)/(2p)}$.
\item[{\bf\small[P-E]}] Consider $\wSo_j\asymp j^{-2p}$ and $\Ev_j^2\asymp
  \exp (-j^{2a}+1)$, with $p> 0$ and $a>0$ then $\oeDi\asymp |\log
  \ObSoNoL -\frac{2p+(2a-1)_{+}}{2a}(\log|\log \ObSoNoL|)|^{1/(2a)}$
  and $\oeRa\asymp|\log \ObSoNoL|^{-p/a}$.
\end{itemize}
In all three cases Assumption \ref{tr:as:mi} and \eqref{tr:as:mi:ac} hold true.\qed
\end{illu}


%% file: _4Adaptive.tex
\section{Data-driven Bayesian estimation}\label{s:ad}
We will derive in this section a concentration rate given the
aforementioned hierarchical prior distribution. For this purpose
we impose additional conditions on the behaviour of the sequence
$\Ev=\suite{\Ev}$.
\begin{ass}\label{ad:as:ev}
There  exist finite constants $C_\Ev\geq 1$ and $L_\Ev\geq1$ such that
for all $k,l\in\Nz$ hold
\begin{inparaenum}[\dr\upshape(i\upshape)]
\item\label{ad:as:ev:a} $\max_{j>k}\Ev_j^2\leq C_\Ev \min_{1\leq j\leq k}\Ev_j^2=C_\Ev
\mEvs[k]^{-1}$;
\item\label{ad:as:ev:b}$\mEvs[kl]\leq\mEvs[k]\mEvs[l]$;
\item\label{ad:as:ev:c}$1\leq \mEvs[k]/\oEvs[k]\leq L_\Ev$.
\end{inparaenum}
\end{ass}%
\noindent We may emphasise that Assumption \ref{ad:as:ev} \eqref{ad:as:ev:a} holds
trivially with $C_\Ev=1$ if the sequence $\Ev$ is monotonically
decreasing. Moreover, considering the typical choices of the sequence
$\Ev$ presented in Illustration \ref{tr:il:mi}, Assumption
\ref{ad:as:ev} \eqref{ad:as:ev:b} and \eqref{ad:as:ev:c} hold only true
in case of a polynomial decay, i.e., {\bf\small[P-P]} and
{\bf\small[E-P]}. In other words,  Assumption
\ref{ad:as:ev} excludes an exponential decay of $\Ev$, i.e., {\bf\small[P-E]}.

\begin{ass}\label{ad:as:tr}
Let $\RvSoEr$, $\TrSo$ and $\Ev$ be sequences such that
\begin{equation}
0<\kappa^{\TrSy}:=\kappa^{\TrSy}(\RvSoEr,\TrSo,\Ev):=\inf_{0<\ObSoNoL<\ObSoNoL_o}\set{(\treRa)^{-1}[\gb_{\treDi}\wedge \ObSoNoL\,
\treDi\, \oEvs[\treDi]]}\leq 1.\label{ad:as:tr:e1}
\end{equation}
\end{ass}%
\noindent Observe that   $\gb_{\treDi}\geq
\kappa^{\TrSy}\treRa>0$ due to Assumption \ref{ad:as:tr} which in turn implies $\gb_k>0$ for all
$k\in\Nz$ and, hence $\treDi\to\infty$ as
$\ObSoNoL\to0$.  Indeed, if there exists $K\in\Nz$ such
that $\gb_K=0$ and $\gb_{K-1}>0$ then there
exists $\ObSoNoL_{\TrSy}\in(0,1)$ with
$\ObSoNoL_{\TrSy}K\oEvs[K]<\gb_{K-1}$ and for all
$\ObSoNoL\in(0,\ObSoNoL_{\TrSy})$ it is easily seen that
$\treDi=K$ and hence $\gb_{\treDi}=0$.
Moreover, due to Assumption
\ref{ad:as:ev} \eqref{ad:as:ev:c} there exists a constant $L_\Ev$
depending only on $\Ev$ such that  $\ObSoNoL\,\treDi\mEvs[\treDi] (\treRa)^{-1}\leq
\mEvs[\treDi](\oEvs[\treDi])^{-1}\leq L_{\Ev}$, i.e.,  condition
\eqref{tr:as:pv:ac} holds true uniformly for all parameters
$\So\in\Hspace$. If we suppose in addition to Assumption \ref{ad:as:ev}
and \ref{ad:as:tr} that the sequence
of prior variances meets Assumption \ref{tr:as:pv} and that
$\treDi\to\infty$ as $\ObSoNoL\to0$, then the
assumptions of Theorem \ref{tr:pr:ora} are satisfied and $\treRa$
provides up to a constant an upper and lower bound of  the posterior
concentration rate associated with the oracle prior sub-family
$\{P_{\treDi}\}_{\treDi}$.

\noindent  Let us specify the prior distribution $P_{\RvDi}$
of the thresholding parameter $\RvDi$ taking its values in
$\{1,\dotsc,\DiMa\}$ with $\DiMa$ as in Assumption \ref{tr:as:pv},
and for $1\leq
 m\leq \DiMa$
 \begin{equation} \label{ad:de:pr:di}
p_{\RvDi}(\Di):=P_{\RvDi}(\RvDi=\Di)=\frac{\exp(-3 C_\Ev
  \Di/2)\prod_{j=1}^\Di(\RvSoVa_j/\sigma_j)^{1/2}}{\sum_{k=1}^{\DiMa}\exp(-3
  C_\Ev k/2)\prod_{j=1}^{k}(\RvSoVa_j/\sigma_j)^{1/2}}.
\end{equation}
\noindent Keeping in mind the sequences $\So^{\ObSo}=\suite{\So^{\ObSo}}$ and $\sigma=\suite{\sigma}$ of  conditional means and
variances, respectively,  given by $\So^{\ObSo}_j=\sigma_j(\Ev_j\ObSoNoL^{-1}Y_j+
\RvSoVa^{-1}_j\RvSoEr_j)$ and
$\sigma_j=(\RvSoVa_j^{-1}+\Ev_j^2\ObSoNoL^{-1})^{-1}$, for each
$\Di\in\Nz$ the sequence $\hSo^\Di=\suite{\hSo^\Di}=\Ex[\RvSo^\Di|\ObSo]$ of  posterior
means of $\RvSo^\Di$ satisfies  $\hSo^\Di_j=\So^{\ObSo}_j \Ind{\set{1\leq j\leq \Di}}+
\RvSoEr_j \Ind{\set{j> \Di}}$. Introducing further  the weighted norm
$\normV{\theta}_{\sigma}^2:=\sum_{j\geq1}\theta_j^2/\sigma_j$ for $\theta\in\Hspace$ the
posterior distribution $P_{\RvDi|\ObSo}$ of the thresholding parameter
$\RvDi$  is  given by
 \begin{equation} \label{ad:po:di}
p_{\RvDi|\ObSo}(m)=P_{\RvDi|\ObSo}(\RvDi=m)
=\frac{\exp(-\frac{1}{2}\{-\normV{\hSo^\Di-\RvSoEr}_{\sigma}^2+3
  C_\Ev m\})}{\sum_{k=1}^{\DiMa}\exp(-\frac{1}{2}\{-\normV{\hSo^k-\RvSoEr}_{\sigma}^2+3
  C_\Ev k\})}
\end{equation}
\noindent Interestingly, the posterior distribution $P_{\RvDi|\ObSo}$ of the thresholding
parameter $\RvDi$ is concentrating around the oracle dimension
parameter $\treDi$ as $\ObSoNoL$ tends to zero. To be more precise,
there exists $\ObSoNoL_{\TrSy}\in (0,1)$ such that $\treDi\leq \DiMa$
for all $\ObSoNoL\in (0,\ObSoNoL_{\TrSy})$ since $\treRa=o(1)$ for
$\ObSoNoL\to0$. Let us further define for all $\ObSoNoL\in (0,\ObSoNoL_{\TrSy})$
\begin{multline}\label{ad:de:mp}
\meDi:=\min\set{m\in\set{1,\dotsc,\treDi}: \gb_\Di\leq 8 L_{\Ev}C_{\Ev}(1+1/d)\treRa}\quad\text{and}\\\peDi:=\max\set{m\in\set{\treDi,\dotsc,\DiMa}: \Di\leq 5L_{\Ev}(\ObSoNoL\mEvs[\treDi])^{-1} \treRa }
\end{multline}
where  the defining sets are not empty under  Assumption
\ref{ad:as:ev} since $8  L_{\Ev}C_{\Ev}(1+1/d)\treRa\geq 8
  L_{\Ev}C_{\Ev}(1+1/d)\gb_{\treDi}\geq \gb_{\treDi}$ and  $5L_{\Ev}(\ObSoNoL\mEvs[\treDi])^{-1} \treRa \geq
 5\treDi\geq \treDi$. Moreover, under Assumption \ref{ad:as:tr} it is
 easily verified that $\meDi\to \infty$ as
 $\ObSoNoL\to 0$.
\begin{lem}\label{ad:le:pm}
If Assumptions \ref{tr:as:pv} and \ref{ad:as:ev}  hold true then for all $\ObSoNoL\in(0,\ObSoNoL_{\TrSy})$
\begin{enumerate}[label=\emph{\textbf{(\roman*)}},ref=\emph{\textbf{(\roman*)}}]
\item\label{ad:le:pm:i} $\Ex_{\TrSo}P_{\RvDi|\ObSo}(1\leq
  \RvDi<\meDi)\leq
  2\exp\big(-\frac{7C_{\Ev}}{32}\treDi+\log\DiMa\big)\leq
  2\exp\big(-\frac{C_{\Ev}}{5}\treDi+\log\DiMa\big)$;
\item\label{ad:le:pm:ii}
  $\Ex_{\TrSo}P_{\RvDi|\ObSo}(\peDi<
  \RvDi\leq\DiMa)\leq 2\exp\big(-\frac{4C_\Ev}{9}\treDi+\log\DiMa\big)\leq
  2\exp\big(-\frac{C_{\Ev}}{5}\treDi+\log\DiMa\big)$.
\end{enumerate}
\end{lem}
\noindent Recall that $\treDi\to\infty$ as $\ObSoNoL\to0$ under Assumption
\ref{ad:as:tr}. If in addition $\treDi/(\log
\DiMa)\to\infty$ as  $\ObSoNoL\to0$ then Lemma \ref{ad:le:pm} states that the posterior distribution of
the thresholding parameter $\RvDi$ is vanishing outside the set
$\{\meDi,\dotsc,\peDi\}$ as $\ObSoNoL\to0$. On the other hand side, the posterior distribution $P_{\RvDiRvSo|\ObSo}$
of $\RvDiRvSo=\suite{\RvDiRvSo}$ associated with the hierarchical
prior is a weighted mixture of the posterior distributions
$\{P_{\DiRvSo|\ObSo}\}_{\Di=1}^{\DiMa}$ studied in section \ref{s:tr}, that is,
$P_{\RvDiRvSo|\ObSo}=\sum_{\Di=1}^{\DiMa}p_{\RvDi|\ObSo}(\Di)P_{\DiRvSo|\ObSo}$.
The next assertion shows that considering  posterior distributions
$\{P_{\DiRvSo|\ObSo}\}_{\Di=\meDi}^{\peDi}$ associated with
thresholding parameters belonging to $\{\meDi,\dotsc,\peDi\}$ only, then their
concentration rate equals $\treRa$ up to a constant.
\begin{lem}\label{ad:le:pv}
If Assumptions \ref{tr:as:pv}, \ref{ad:as:ev} and \ref{ad:as:tr} hold
true then  for all $\ObSoNoL\in(0,\ObSoNoL_{\TrSy})$
\begin{enumerate}[label=\emph{\textbf{(\roman*)}},ref=\emph{\textbf{(\roman*)}}]
\item\label{ad:le:pv:i}$\sum_{\meDi\leq\Di\leq\peDi}\Ex_{\TrSo}
  P_{\DiRvSo[\Di]|\ObSo}\big(\HnormV{\DiRvSo[\Di]-\TrSo}^2>K^o\treRa\big)\leq 74 \exp(-\meDi/36)$;
\item\label{ad:le:pv:ii}$\sum_{\meDi\leq\Di\leq\peDi}\Ex_{\TrSo}
  P_{\DiRvSo[\Di]|\ObSo}\big(\HnormV{\DiRvSo[\Di]-\TrSo}^2<(K^o)^{-1}\treRa
  \big)\leq 4 (K^{\TrSy})^2 \exp(-\meDi/(K^{\TrSy})^2)$,
\end{enumerate}
where $K^{\TrSy}:=10 ((1+1/d)\vee \HnormV{\TrSo-\RvSoEr}^2/d^{2})L^2_\Ev(8C_\Ev(1+1/d)\vee D^{\TrSy}\mEvs[D^{\TrSy}])$ with $D^{\TrSy}:=D^{\TrSy}(\RvSoEr,\TrSo,\Ev):=\ceil{5 L_\Ev/\kappa^{\TrSy}}$.
\end{lem}
\noindent From Lemma \ref{ad:le:pm} and \ref{ad:le:pv} we derive next
upper and lower bounds for the concentration rate of the  posterior distribution $P_{\RvDiRvSo|\ObSo}$
by decomposing the weighted mixture into three parts with respect to
$\meDi$ and $\peDi$  which we bound separately.
\begin{theo}[Oracle posterior concentration rate]\label{ad:th:ora}
Let Assumptions \ref{tr:as:pv}, \ref{ad:as:ev} and \ref{ad:as:tr} hold
true. If in addition $(\log \DiMa)/\treDi\to0$ as
$\ObSoNoL\to0$, then
\[\lim_{\ObSoNoL\to0}\Ex_{\TrSo}P_{\RvDiRvSo|\ObSo}(
 (K^{\TrSy})^{-1} \treRa\leq \HnormV{\RvDiRvSo-\TrSo}^2\leq K^{\TrSy} \treRa)=1\]
where $K^{\TrSy}$ is given in Lemma \ref{ad:le:pv}.
\end{theo}
\noindent We shall emphasise that the Bayes
estimator $\hSo:=\suite{\hSo}:=\Ex[\RvDiRvSo|\ObSo]$ associated with the
hierarchical prior and given by
$\hSo_j=\RvSoEr_j$ for $j>\DiMa$ and $\hSo_j=\RvSoEr_j\,P(1\leq \RvDi<
j\vert\ObSo) + \So_j^{\ObSo}\,P(j\leq \RvDi\leq \DiMa\vert\ObSo)$ for
all $1\leq j\leq \DiMa$, does not take into account any
prior information related to the parameter of interest, and hence it
is fully data-driven. The next assertion provides an upper bound of
its MISE.
\begin{theo}[Oracle optimal Bayes estimator]\label{ad:th:bo}
Under Assumptions \ref{tr:as:pv}, \ref{ad:as:ev} and \ref{ad:as:tr}  consider the Bayes estimator
$\hSo:=\Ex[\RvDiRvSo|\ObSo]$. If  in addition $\log
(\DiMa/\treRa)/\treDi\to0$  as $\ObSoNoL\to0$, then there exists a
constant $K^{\TrSy}:=K^{\TrSy}(\TrSo,\RvSoEr,\Ev,
d, L)<\infty$ such that $\Ex_{\TrSo} \HnormV{\hSo-\TrSo}^2\leq K^{\TrSy}
\treRa$  for all $\ObSoNoL\in(0,\ObSoNoL_{\TrSy})$.
\end{theo}
Both Theorems, \ref{ad:th:ora} and \ref{ad:th:bo} hold true only under
Assumption \ref{ad:as:tr}, which we have seen before imposes an
additional restriction on the parameter of
interest $\TrSo$, i.e., its components differ from the components of the prior mean $\RvSoEr$
infinitely many times. However, for all parameters of interest
satisfying Assumption \ref{ad:as:tr}, the hierarchical prior sequence
allows to recover the oracle posterior concentration rate and the
fully data driven Bayes estimator attains the oracle rate.  {In the last part
of this section we show that  for all $\TrSo\in\cwrSo$  the
 posterior concentration rate and the MISE of the Bayes estimator
 associated with the hierarchical prior are bounded  from above  by the
minimax rate $\oeRa$ up to a constant. In other words, the fully data-driven
hierarchical prior and the  associated Bayes estimator are
minimax-rate optimal.}

Recall the definition \eqref{tr:de:oe} of $\oeDi$ and
  $\oeRa$.  Consider the prior distribution
$P_{\RvDi}$ of the thresholding parameter $\RvDi$, and observe that
there exists $\ObSoNoL_\star$ such that $\oeDi\leq \DiMa$  for all
$\ObSoNoL\in(0,\ObSoNoL_\star)$ since $\oeRa=o(1)$ as $\ObSoNoL\to0$.
Remark that $\ObSoNoL\,\oeDi\mEvs[\oeDi] (\oeRa)^{-1}\leq
\mEvs[\oeDi](\oEvs[\oeDi])^{-1}\leq L_{\Ev}$ with $L_\Ev$
depending only on $\Ev$ due to Assumption
\ref{ad:as:ev} \eqref{ad:as:ev:c}, i.e.,  condition \eqref{tr:as:mi:ac} holds true uniformly for all parameters
$\So\in\Hspace$.  If we assume in addition that the sequence
of prior variances satisfies Assumption \ref{tr:as:pv} and that Assumption
\ref{tr:as:mi} holds true, then the
conditions of Theorem \ref{tr:th:opc} are satisfied and $\oeRa$
provides up to a constant an upper and lower bound of  the posterior
concentration rate associated with the minimax prior sub-family $\{P_{\oeDi}\}_{\oeDi}$. On the other hand side, the
posterior distribution $P_{\RvDi|\ObSo}$ of the thresholding
parameter $\RvDi$ is concentrating around the minimax-optimal dimension
parameter $\oeDi$ as $\ObSoNoL$ tends to zero. To be more precise, for
$\ObSoNoL\in(0,\ObSoNoL_\star)$ let us define
\begin{multline}\label{ad:de:omp}
\moeDi:=\min\set{m\in\set{1,\dotsc,\oeDi}: \gb_\Di\leq 8 L_{\Ev}C_{\Ev}(1+1/d)(1\vee\rSo)\oeRa}\quad\text{and}\\\poeDi:=\max\set{m\in\set{\oeDi,\dotsc,\DiMa}: \Di\leq 5L_{\Ev}(\ObSoNoL\mEvs[\oeDi])^{-1}(1\vee\rSo) \oeRa }
\end{multline}
where  the defining sets are not empty under  Assumption
\ref{ad:as:ev} since $8  L_{\Ev}C_{\Ev}(1+1/d)(1\vee\rSo)\oeRa\geq 8
L_{\Ev}C_{\Ev}(1+1/d)\rSo\wSo_{\oeDi}\geq 8
  L_{\Ev}C_{\Ev}(1+1/d)\gb_{\oeDi}\geq \gb_{\oeDi}$ and  $5L_{\Ev}(\ObSoNoL\mEvs[\oeDi])^{-1} (1\vee\rSo)\oeRa \geq
 5\oeDi\geq \oeDi$.  Moreover, it is again straightforward to see that $\moeDi\to \infty$ as
 $\ObSoNoL\to 0$.
\begin{lem}\label{ad:le:opm}
If Assumption \ref{tr:as:pv} and \ref{ad:as:ev} hold true then for all
$\TrSo\in\cwrSo$ and $\ObSoNoL\in(0,\ObSoNoL_\star)$
\begin{enumerate}[label=\emph{\textbf{(\roman*)}},ref=\emph{\textbf{(\roman*)}}]
\item\label{ad:le:opm:i} $\Ex_{\TrSo}P_{\RvDi|\ObSo}(1\leq  \RvDi<\moeDi)
\leq
  2\exp\big(-\frac{C_{\Ev}(1\vee\rSo)}{5}\oeDi+\log\DiMa\big)$;
\item\label{ad:le:opm:ii}
  $\Ex_{\TrSo}P_{\RvDi|\ObSo}(\poeDi<
  \RvDi\leq\DiMa)
  \leq
  2\exp\big(-\frac{C_{\Ev}(1\vee\rSo)}{5}\oeDi+\log\DiMa\big)$.
\end{enumerate}
\end{lem}
\noindent By employing  Lemma \ref{ad:le:opm} we show next for
  each $\TrSo\in\cwrSo$ that
  the minimax rate $\oeRa$ provides up to a constant an upper bound
 for the posterior concentration rate associated with the fully
 data-driven hierarchical prior distribution $P_{\RvDiRvSo}$.
\begin{theo}[Minimax optimal posterior concentration rate]\label{ad:th:ra:mm}
Let Assumption \ref{tr:as:pv}, \ref{tr:as:mi} and \ref{ad:as:ev} hold
true. If in addition $(\log \DiMa)/\oeDi\to0$ as
$\ObSoNoL\to0$, then
\begin{enumerate}[label=\emph{\textbf{(\roman*)}},ref=\emph{\textbf{(\roman*)}}]
\item\label{ad:th:ra:mm:i} for all $\TrSo\in\cwrSo$ we have
\[\lim_{\ObSoNoL\to0}\Ex_{\TrSo}P_{\RvDiRvSo|\ObSo}(\HnormV{\RvDiRvSo-\TrSo}^2\leq K^{\OpSy} \oeRa)=1\]
where $K^{\OpSy}:={16} ((1+1/d)\vee \rSo/d^{2})L^2_\Ev(8C_\Ev(1+1/d)\vee
D^{\OpSy}\mEvs[D^{\OpSy}]){(1\vee\rSo)}$ with
$D^{\OpSy}:=D^{\OpSy}(\wSo,\Ev):=\ceil{5 L_\Ev/\kappa^{\OpSy}}$;
\item\label{ad:th:ra:mm:ii} for any monotonically increasing and
  unbounded sequence $(K_\ObSoNoL)_{\ObSoNoL}$  holds
\[\lim_{\ObSoNoL\to0}\inf_{\TrSo\in\cwrSo}\Ex_{\TrSo}P_{\RvDiRvSo|\ObSo}(\HnormV{\RvDiRvSo-\TrSo}^2\leq
K_\ObSoNoL \oeRa)=1.\]
\end{enumerate}
\end{theo}
\noindent We shall emphasise that due to Theorem \ref{ad:th:ora} for all $\TrSo\in\cwrSo$ satisfying
Assumption \ref{ad:as:tr} the posterior concentration rate associated
with the hierarchical prior attains the oracle rate $\treRa$ which
might be far smaller than the minimax-rate $\oeRa$. Consequently, the
minimax rate cannot provide an uniform  lower bound  over $\cwrSo$ for the posterior concentration rate associated with the
hierarchical prior. However, due to Theorem  \ref{ad:th:ra:mm} the
posterior concentration rate is for all $\TrSo\in\cwrSo$, independently 
that Assumption \ref{ad:as:tr} holds, at least of
the order of the minimax rate $\oeRa$. The next assertion establishes the minimax-rate optimality of the
fully data-driven Bayes estimator.
\begin{theo}[Minimax optimal Bayes estimate]\label{ad:th:be:mm}
Under Assumption \ref{tr:as:pv}, \ref{tr:as:mi} and \ref{ad:as:ev}  consider the Bayes estimator
$\hSo:=\Ex[\RvDiRvSo|\ObSo]$. If in addition $\log(
\DiMa/\oeRa)/\oeDi\to0$  as $\ObSoNoL\to0$, then there exists $K^\star:=K^\star(\cwrSo,\Ev,d)<\infty$ such that $\sup_{\TrSo\in\cwrSo}\Ex_{\TrSo} \HnormV{\hSo-\TrSo}^2\leq K^\star
\oeRa$  for all $\ObSoNoL\in(0,\ObSoNoL_\star)$.
\end{theo}
\noindent Let us briefly comment on the last assertion by considering
again the improper
specification of the prior family
$\{P_{\RvSo^{\Di}}\}_{\Di}$ introduced in Section \ref{s:mo}. Recall
that in this situation for each
$\Di\in\Nz$ the Bayes estimator $\hSo^\Di=\Ex[\RvSo^\Di|\ObSo]$ of
$\RvSo^\Di$ equals an orthogonal projection estimator, i.e.,
$\hSo^\Di=(Y/\Ev)^m$. Moreover, the
posterior probability of the thresholding parameter
$\RvDi$ taking a value $m\in\{1,\dotsc,\DiMa\}$ is  proportional to $\exp(-\frac{1}{2}\{-\normV{(Y/\Ev)^m}_{\ObSoNoL\Evs}^2+3
  C_\Ev m\})$, and hence the data-driven Bayes estimator
  $\hSo=\suite{\hSo}=\Ex[\RvDiRvSo|\ObSo]$  equals the shrinked
  orthogonal projection estimator given by
 \begin{equation*}
\hSo_j
=\frac{\sum_{\Di=j}^{\DiMa}\exp(-\frac{1}{2}\{-\normV{(Y/\Ev)^\Di}_{\ObSoNoL\Evs}^2+3
  C_\Ev m\})}{\sum_{\Di=1}^{\DiMa}\exp(-\frac{1}{2}\{-\normV{(Y/\Ev)^\Di}_{\ObSoNoL\Evs}^2+3
  C_\Ev \Di\})}\times \frac{Y_j}{\Ev_j} \Ind{\set{1\leq j\leq
      \DiMa}}.
\end{equation*}
From Theorem \ref{ad:th:be:mm} it follows now, that the fully data-driven shrinkage estimator $\hSo$ is
minimax-optimal up to a constant for a wide variety of parameter
spaces $\cwrSo$ provided Assumptions \ref{tr:as:mi} and \ref{ad:as:ev}
hold true. Interestingly, identifying
$\Upsilon(\hSo^\Di):=-(1/2)\normV{(Y/\Ev)^\Di}_{\ObSoNoL\Evs}^2$ as a
contrast and
$\pen_\Di:=3/2 C_\Ev \Di$ as a penalty term the $j$-th shrinkage
weight is proportional to $\sum_{\Di=j}^{\DiMa}\exp(-\{\Upsilon(\hSo^\Di)+\pen_m\})$. Roughly
speaking, in comparison to a classical model selection approach where
a data-driven estimator $\hSo^{\whm}=(Y/\Ev)^{\whm}$  is obtained by
selecting the dimension parameter $\whm$ as minimum of a
penalised contrast criterion over a
class of admissible models $\{1,\dotsc,\DiMa\}$, i.e.,
$\whm=\argmin\nolimits_{1\leq m\leq \DiMa}\{\Upsilon(\hSo^\Di)+\pen_m\}$, following the
Bayesian approach each of the
$\DiMa$ components of the data-driven Bayes estimator  is shrunk 
proportional to the associated values of the  penalised contrast criterion.


%% file: _Proof_3Rate.tex
\begin{proof}[\noindent{\color{darkred}\sc Proof of Lemma \ref{tr:le:te}.}]\label{tr:le:te:pr} Let $X_j= \beta_jZ_j+\alpha_j$ with
independent and standard normally
distributed random variables $\{Z_j\}_{j=1}^m$.
 We start our proof with the observation that  $\Ex(S_{m})=\sum_{j=1}^{m}\set{\beta_j^2+\alpha_j^2}$ and define
  $\Sigma_m:=\frac{1}{2}\sum_{j=1}^{m}\Var\big((\beta_jZ_j+\alpha_j)^2\big)=\sum_{j=1}^{m}\beta_j^2(\beta_j^2+2\alpha_j^2)$. Let $t_m:=\max_{1\leq j\leq m}\beta_j^2$ and  by  using that $
  v_m\geq \sum_{j=1}^m \beta_j^2$
  and $
  r_m\geq \sum_{j=1}^m \alpha_j^2$ we have $\Ex(S_{m})\leq v_m+r_m $
  and $\Sigma_m\leq t_m\;(v_m+2r_m)$. These bounds are used below without
  further reference. There exist several results of tail bound for sums of independent
squared Gaussian random variables and we present next a version which
is due to \cite{Birge2001} and can be shown following the lines of the proof of Lemma 1 in \cite{LaurentLM2012}. For all $x>0$ we have
\begin{multline}\label{tr:le:te:pr:e1}
  P(S_m -\Ex S_m\geq2\sqrt{\Sigma_mx}+2 t_mx)\leq
  \exp(-x)\qquad\mbox{and}\\\, P(S_m -\Ex S_m\leq-2\sqrt{\Sigma_mx})\leq
  \exp(-x).
\end{multline}
Consider \eqref{tr:le:te:e2}. Keeping in mind that for all $c\geq0$,
$(3/2)c(v_m + 2r_m)\geq c(v_m + 2r_m) + 2t_m c(c\wedge1)(v_m +
2r_m)/(4t_m)$ 
 and $(c\vee1)t_m(v_m+2r_m)\geq \Sigma_m$ we conclude
for $x:={c(c\wedge1)(v_m + 2r_m)}/{(4t_m)}$  that  $(3/2) c(v_m +
2r_m)\geq 2\sqrt{ \Sigma_m x} + 2 t_m x$ 
  and hence by employing the first exponential bound in
  \eqref{tr:le:te:pr:e1} we obtain \eqref{tr:le:te:e2}. 
  On the other hand side, since  $c(v_m +
  2r_m)\geq 2\sqrt{\Sigma_mx}$ for all $c\geq0$ 
assertion  \eqref{tr:le:te:e1} follows by employing the second exponential bound in
  \eqref{tr:le:te:pr:e1}, which completes the proof. 
\end{proof}
\begin{proof}[\noindent{\color{darkred}\sc Proof of Proposition \ref{tr:pr:pbm}.}]\label{tr:pr:pbm:pr}
 We intend to  apply the technical Lemma \ref{tr:le:te}. Consider first the assertion \eqref{tr:pr:pbm:e1}. Let  $s_\Di$
 and $c_1$ be positive constants (to be specified below). {
 Keeping in mind that the posterior distribution of
 $\DiRvSo_j$ given $\ObSo_j$ is degenerated on $\RvSoEr_j$ for $j>\Di$ and} that $ \gb_\Di=
 \sum_{j>\Di}(\TrSo[j]-\RvSoEr_j)^2$ we have
\begin{multline*}
\Ex_{\TrSo}P_{\DiRvSo|\ObSo}\bigg(\HnormV{\DiRvSo-\TrSo}^2> \gb_\Di+ \oPoVa+\frac{3c_1}{2}\Di\mPoVa+(3c_1+1) s_\Di\bigg)\\
 =\; \Ex_{\TrSo}P_{\DiRvSo|\ObSo}\bigg(
 \sum_{j=1}^{\Di} (\DiRvSo_j-\TrSo[j])^2 >\oPoVa+ \frac{3c_1}{2}\Di\mPoVa + (3c_1+1) s_\Di \bigg).
\end{multline*}
  Define
 $S_\Di^{\DiRvSo}:=\sum_{j=1}^{\Di}(\DiRvSo_j-\TrSo[j])^2$ {
 where conditional on $\ObSo$ the random variables
 $\{\DiRvSo_j-\TrSo[j]\}_{j=1}^\Di$ are independent and normally distributed with conditional mean
 $\So^{\ObSo}_j-\TrSo[j]$ and conditional variance $\sigma_j$.}  Observe that  $
 \oPoVa= \sum_{j=1}^{\Di}\sigma_j$ and
 $\Ex_{\DiRvSo|\ObSo}[S_\Di^{\DiRvSo}]=\oPoVa+\sum_{j=1}^{\Di}(\So^{\ObSo}_j-\TrSo[j])^2$. Introduce   the event
$
\Omega_{\Di}:=\set{{
 \sum_{j=1}^{\Di}(\So^{\ObSo}_j-\TrSo[j])^2
      \leq s_\Di}}
$
where obviously $\Ind{\Omega_{\Di}}\Ex_{\DiRvSo|\ObSo}[S_\Di^{\DiRvSo} ]\leq
    \oPoVa+s_\Di$ and hence,
 \begin{multline*}
   \Ex_{\TrSo}\Ind{\Omega_{\Di}}P_{\DiRvSo|\ObSo}\bigg(S_m^{\DiRvSo} > \oPoVa+\frac{3c_1}{2}\Di\mPoVa + (3c_1+1)s_\Di \bigg)\\\leq
   \Ex_{\TrSo}\Ind{\Omega_{\Di}}P_{\DiRvSo|\ObSo}\bigg(S_m^{\DiRvSo}-\Ex_{\DiRvSo|\ObSo}[S_\Di^{\DiRvSo}] > \frac{3c_1}{2}(\Di\mPoVa + 2s_\Di) \bigg).
 \end{multline*}
Employing \eqref{tr:le:te:e2} in Lemma \ref{tr:le:te} we bound the left hand side in the last display
and we obtain
 \begin{displaymath}
   \Ex_{\TrSo}\Ind{\Omega_{\Di}}P_{\DiRvSo|\ObSo}\bigg(S_\Di^{\DiRvSo} > \oPoVa+\frac{3c_1}{2}\Di\mPoVa + (3c_1+1)s_\Di
   \bigg)
\leq
\exp(-\frac{c_1(c_1\wedge1)(\Di\mPoVa + 2s_\Di)}{4\mPoVa})
 \end{displaymath}
where we used that  $
\Di\mPoVa\geq \sum_{j=1}^{\Di}\sigma_j$ for $\mPoVa = \max_{1\leq j\leq
    \Di}\sigma_j$. As a consequence,
\begin{multline}\label{tr:pr:pbm:pr:e1}
\Ex_{\TrSo}P_{\DiRvSo|\ObSo}(\HnormV{\DiRvSo-\TrSo}^2> \gb_\Di+ \oPoVa+\frac{3c_1}{2}\Di\mPoVa+(3c_1+1) s_\Di)\\
\leq
\exp(-\frac{c_1(c_1\wedge1)(\Di\mPoVa + 2s_\Di)}{4\mPoVa}) + P_{\TrSo}({\Omega_{\Di}^c}).
\end{multline}
 In the following, we bound the remainder probability of the event
 $\Omega_{\Di}^c=\set{S_\Di^{\ObSo} > s_\Di}$ for $S_\Di^{\ObSo}:=\sum_{j=1}^{\Di}(\So^{\ObSo}_j-\TrSo[j])^2$
where the random variables $\{\So^{\ObSo}_j-\TrSo[j]\}_{j=1}^\Di$ are independent and
normally distributed with mean $\Ex_{\TrSo}[\So^{\ObSo}_j]-\TrSo[j]$ and standard deviation
$\beta_j:={\epsilon^{1/2}}{\Ev_j^{-1}}\mu_j$ for $\mu_j:=({\ObSoNoL\Ev_j^{-2}\RvSoVa^{-1}_j+1})^{-1}$. Since
$\sigma_j= \epsilon{\Ev_j^{-2}}\mu_j$ and $\mu_j\leq 1$ if follows that  $
\oPoVa\geq \sum_{j=1}^{\Di}\beta_j^2$ and   $
\mPoVa \geq \max_{1\leq j\leq  \Di}\beta_j^2$. Moreover, $\gr_\Di=
\sum_{j=1}^{\Di}(\Ex_{\TrSo}[\So^{\ObSo}_j]-\TrSo[j])^2$  and hence
  $\Ex_{\TrSo}[S_\Di^{\ObSo}]\leq\oPoVa+\gr_\Di$. Denote $s_\Di:= \oPoVa+\frac{3c_2}{2}\Di\mPoVa + (3c_2+1) \gr_\Di$
  which allows us to write
\begin{multline*}
 P_{\TrSo}({\Omega_{\Di}^c})=  P_{\TrSo}\bigg(S_\Di^{\ObSo} > \oPoVa+\frac{3c_2}{2}\Di\mPoVa + (3c_2+1)\gr_\Di \bigg)\\\leq
P_{\TrSo}\bigg(S_\Di^{\ObSo}-\Ex_{\TrSo}[S_\Di^{\ObSo}] > \frac{3c_2}{2}(\Di\mPoVa + 2\gr_\Di) \bigg)
 \end{multline*}
 The right hand side in the last display is bounded by  employing \eqref{tr:le:te:e2} in Lemma
 \ref{tr:le:te}, and hence
\begin{equation}\label{tr:pr:pbm:pr:e2}
 P_{\TrSo}({\Omega_{\Di}^c})\leq \exp(-\frac{c_2(c_2\wedge1)(\Di\mPoVa + 2\gr_\Di)}{4\mPoVa}).
 \end{equation}
By combination of \eqref{tr:pr:pbm:pr:e1}, \eqref{tr:pr:pbm:pr:e2} and $s_\Di= \oPoVa+\frac{3c_2}{2}\Di\mPoVa +
(3c_2+1) \gr_\Di$ it follows that
\begin{multline*}
 \Ex_{\TrSo}P_{\DiRvSo|\ObSo}\big( \HnormV{\DiRvSo-\TrSo}^2> \gb_\Di+ \oPoVa+\frac{3c_1}{2}\Di\mPoVa+(3c_1+1)[\oPoVa+\frac{3c_2}{2}\Di\mPoVa + (3c_2+1) \gr_\Di]\big)\\
\leq \exp(-\frac{c_1(c_1\wedge1)(3c_2+1)( \Di\mPoVa + 2 \gr_\Di)}{4\mPoVa})+ \exp(-\frac{c_2(c_2\wedge1)(\Di\mPoVa + 2\gr_\Di)}{4\mPoVa})
\end{multline*}
The assertion \eqref{tr:pr:pbm:e1} follows now by taking  $
c_1=1/3=c_2$. The proof of the assertion \eqref{tr:pr:pbm:e2} follows along the lines of the proof of \eqref{tr:pr:pbm:e1}.  Let  $c_3$ be a positive
constant (to be specified below).  Since $\Ex_{\DiRvSo|\ObSo}[S_\Di^{\DiRvSo} ]\geq
    \oPoVa$ it trivially follows from  \eqref{tr:le:te:e1} in Lemma \ref{tr:le:te}  that
\begin{multline*}
  \Ex_{\TrSo}\1_{\Omega_{\Di}}P_{\DiRvSo|\ObSo}\bigg(S_m^{\DiRvSo} <
  \oPoVa-c_3\Di\mPoVa -2c_3 s_\Di \bigg)
\\\hfill\leq
\Ex_{\TrSo}\1_{\Omega_{\Di}}P_{\DiRvSo|\ObSo}\bigg(S_\Di^{\DiRvSo} - \Ex_{\DiRvSo|\ObSo}[S_\Di^{\DiRvSo}]<
 -c_3(\Di\mPoVa+ 2s_\Di) \bigg)
\\\leq
\exp(-\frac{c_3(c_3\wedge1)(\Di\mPoVa + 2s
_\Di)}{4\mPoVa})
\end{multline*}
Combining the last bound, the estimate \eqref{tr:pr:pbm:pr:e2} and  $
\gb_\Di=\sum_{j>\Di}(\TrSo[j]-\RvSoEr_j)^2$ it
follows that
\begin{multline*}
 \Ex_{\TrSo}P_{\DiRvSo|\ObSo}\bigg( \HnormV{\DiRvSo-\TrSo}^2< \gb_\Di+\oPoVa -c_3\Di\mPoVa-2 c_3[\oPoVa+\frac{3c_2}{2}\Di\mPoVa
 + (3c_2+1) \gr_\Di]\bigg)
 \\\hfill\leq \Ex_{\TrSo}\1_{\Omega_{\Di}}P_{\DiRvSo|\ObSo}(
 S_\Di^{\DiRvSo}
 <
 \oPoVa-c_3\Di\mPoVa-2 c_3 s_\Di)+P_{\TrSo}(\Omega_{\Di}^c)\\
\leq \exp(-\frac{c_3(c_3\wedge1)(3c_2+1)( \Di\mPoVa + 2 \gr_\Di)}{4\mPoVa})+ \exp(-\frac{c_2(c_2\wedge1)(\Di\mPoVa + 2\gr_\Di)}{4\mPoVa})
\end{multline*}
The assertion \eqref{tr:pr:pbm:e2} follows now by taking $c_2=1/3$  which completes the proof.\end{proof}
\begin{proof}[\noindent{\color{darkred}\sc Proof of Proposition \ref{tr:pr:bc}.}]\label{tr:pr:bc:pr}
Keeping in mind the notations and findings used in the proof of Proposition
\ref{tr:pr:pbm} we have
\begin{multline}\label{tr:pr:bc:pr:e1}
  \Ex_{\TrSo}
  \HnormV{\hSo^{\eDi}-\TrSo}^2=\Ex_{\TrSo}\sum_{j=1}^{\eDi}(\So^{\ObSo}_j-\TrSo[j])^2
  + \sum_{j>\eDi}(\RvSoEr_j-\TrSo[j])^2
  \\=\sum_{j=1}^{\eDi}\PoVa[j](\PoVa[j]\Ev_j^{2}\ObSoNoL^{-1})+\gr_{\eDi}+\gb_{\eDi},
\end{multline}
 which together with $\PoVa[j]\Ev_j^{2}\ObSoNoL^{-1}\leq1$  implies $\Ex_{\TrSo} \HnormV{\hSo^{\eDi}-\TrSo}^2 \leq  \gb_{\eDi}+\oPoVa[\eDi]+\gr_{\eDi}$.
Exploiting the Assumption \ref{tr:as:pr}, that is, $\gr_{\eDi}\leq K [\gb_{\eDi}\vee\oPoVa[\eDi]]$, we obtain the assertion.
\end{proof}
\begin{proof}[\noindent{\color{darkred}\sc Proof of Theorem
    \ref{tr:pr:bo}.}]\label{tr:pr:bo:pr}
The assertion follows from  \eqref{tr:pr:bc:pr:e1}
  given in the proof of Proposition \ref{tr:pr:bc}. Indeed,
 \eqref{tr:pr:bo:i} follows by combination of
\eqref{tr:pr:bc:pr:e1}, $\sum_{j=1}^{\Di}\PoVa[j](\PoVa[j]\Ev_j^{2}\ObSoNoL^{-1})\leq
\ObSoNoL\Di\oEvs$ and  $\gr_{\Di}\leq
d^{-2}\HnormV{\TrSo-\RvSoEr}^2\ObSoNoL\mEvs$ while \eqref{tr:pr:bc:pr:e1},
$\sum_{j=1}^{\Di}\PoVa[j](\PoVa[j]\Ev_j^{2}\ObSoNoL^{-1})\geq (1+1/d)^{-2}
\ObSoNoL\Di\oEvs$
and  $\gr_{\Di}\geq 0$ imply together
\eqref{tr:pr:bo:ii}. Note that these elementary bounds hold due to
Assumption   \ref{tr:as:pv} for all $\ObSoNoL\in(0,\ObSoNoL_o)$ since
$\treRa=o(1)$ as $\ObSoNoL\to0$,  which completes the proof.
\end{proof}
\begin{proof}[\noindent{\color{darkred}\sc Proof of Theorem
    \ref{tr:th:opc}}]\label{tr:th:opc:pr} We start the proof with the
  observation that due to Assumption \ref{tr:as:mi}
  and  \eqref{tr:as:mi:ac} the sub-family $\{P_{\RvSo^{\oeDi}}\}_{\oeDi}$
satisfies the  condition \eqref{tr:as:pv:ac}  uniformly for all
$\TrSo\in\cwrSo$ with $L=L^{\OpSy}/\kappa^{\OpSy}$. Moreover, we have
$\oeRa=o(1)$, as $\ObSoNoL\to0$ and we suppose that Assumption
\ref{tr:as:pv} holds true. Thereby, the
assumptions of  Corollary \ref{tr:co:pb} are satisfied.
From  $((1+1/d)\vee \rSo/d^{2})(L^{\OpSy}/\kappa^{\OpSy})\geq ((1+1/d)\vee
d^{-2}\HnormV{\TrSo-\RvSoEr}^2)L=K$ and the definition of $K^{\OpSy}$ it follows further that
$K^{\OpSy}\geq (4+(11/2)K){(1\vee \rSo)}$
and $(K^{\OpSy})^{-1}\leq
(1/9)(1+1/d)^{-1}\kappa^{\OpSy}$ for all $\TrSo\in\cwrSo$. Moreover, for all $0<\ObSoNoL<\ObSoNoL_o$ we have
$(1\vee\rSo)\,\oeRa\geq
\eRa^{\oeDi}= [\gb_{\oeDi}\vee
    \ObSoNoL\,\oeDi\,\oEvs[\oeDi]]\geq \kappa^{\OpSy}\,\oeRa$.
By combining these elementary inequalities  and  Corollary \ref{tr:co:pb} with $c:=1/(9K)$
and $c\geq 1/K^{\OpSy}$ uniformly  for all $\TrSo\in\cwrSo$  we
obtain for all   $\ObSoNoL\in(0,\ObSoNoL_o)$
\begin{multline}\label{tr:th:opc:pr:e1}
\sup_{\TrSo\in\cwrSo}\Ex_{\TrSo}P_{\DiRvSo[\oeDi]|\ObSo}(
\HnormV{\DiRvSo[\oeDi]-\TrSo}^2> K^{\OpSy}\oeRa)\\
\leq \sup_{\TrSo\in\cwrSo}\Ex_{\TrSo}P_{\DiRvSo[\oeDi]|\ObSo}(
\HnormV{\DiRvSo[\oeDi]-\TrSo}^2> (4+({11}/{2})K)\eRa^{\oeDi})\\\leq
2\exp(-{\oeDi}/{36});
\end{multline}
\begin{multline}\label{tr:th:opc:pr:e2}
\sup_{\TrSo\in\cwrSo}\Ex_{\TrSo}P_{\DiRvSo[\oeDi]|\ObSo}(
\HnormV{\DiRvSo[\oeDi]-\TrSo}^2< (K^{\OpSy})^{-1}\oeRa)\\
\leq \sup_{\TrSo\in\cwrSo}\Ex_{\TrSo}P_{\DiRvSo[\oeDi]|\ObSo}( \HnormV{\DiRvSo[\oeDi]-\TrSo}^2<(1-8\,c\,K)\{(1+1/d)\}^{-1}\eRa^{\oeDi})\\\leq 2\exp(-\oeDi/[2 (K^\OpSy)^2]).
\end{multline}
By combining \eqref{tr:th:opc:pr:e1} and \eqref{tr:th:opc:pr:e2} we
obtain the assertion of the theorem since $\oeDi \to \infty $, which completes the proof.
\end{proof}


%% file: _Proof_4Adaptive.tex
\subsection{Proof of Theorem   \ref{ad:th:ora}}
\begin{proof}[\noindent{\color{darkred}\sc Proof of Lemma
    \ref{ad:le:pm}.}]\label{ad:le:pm:pr}
Consider \ref{ad:le:pm:i}. The claim holds trivially true
in case $\meDi=1$, thus suppose $\meDi>1$ and let
$1\leq \Di<\meDi\leq \treDi$. Define
$S_{\Di}:=\normV{\hSo^{\treDi}-\RvSoEr}_{\sigma}^2-\normV{\hSo^\Di-\RvSoEr}_{\sigma}^2$. Given
an event $\cA_{\Di}$ and its complement
$\cA_{\Di}^c$ (to be specified below) it follows
\begin{multline}\label{ad:le:pm:pr:e1}
 p_{\RvDi|\ObSo}(\Di)=\frac{\exp(\frac{1}{2}\{\normV{\hSo^\Di-\RvSoEr}_{\sigma}^2-3
  C_\Ev m\})}{\sum_{k=1}^{\DiMa}\exp(\frac{1}{2}\{\normV{\hSo^k-\RvSoEr}_{\sigma}^2-3
  C_\Ev k\})}\\
=
\exp\bigg(\frac{1}{2}\big\{-S_{\Di}+3 C_\Ev [\treDi-\Di]\big\}\bigg)\Ind{\cA_{\Di}}
+ \Ind{\cA_{\Di}^c}
\end{multline}
Moreover, elementary algebra  shows
\begin{equation*}
S_{\Di}
=\sum_{j=\Di+1}^{\treDi}\frac{\Ev_j^2\sigma_j}{\ObSoNoL^2}(Y_j-\Ev_j\RvSoEr_j)^2
\end{equation*}
where the random variables $\{\Ev_j\sigma_j^{1/2}\ObSoNoL^{-1}(Y_j-\Ev_j\RvSoEr_j)\}_{j\geq1}$ are independent
and normally distributed with standard deviation $\beta_j=
\Ev_j\sigma_j^{1/2}\ObSoNoL^{-1/2}$ and mean $\alpha_j=
\beta_j\ObSoNoL^{-1/2}\Ev_j(\TrSo[j]-\RvSoEr_j)$. Keeping in mind the
notations used in Lemma \ref{tr:le:te} define
$v_{\Di}:=\sum_{j=\Di+1}^{\treDi}\beta_j^2$ and $r_{\Di}:=\sum_{j=\Di+1}^{\treDi}\alpha_j^2$. We observe that  Assumption
\ref{tr:as:pv}  implies that  $1\geq \beta_j^2\geq (1+1/d)^{-1}$  and
hence it follows by
employing $\min_{\Di<j\leq\treDi}\Ev_j^2\geq \min_{1\leq
  j\leq\treDi}\Ev_j^2=\mEvs[\treDi]^{-1}$ and Assumption
\ref{ad:as:ev} \eqref{ad:as:ev:c}  that
\begin{multline}\label{ad:le:pm:pr:e2}
L_\Ev (\ObSoNoL\mEvs[\treDi])^{-1}\treRa\geq L_\Ev (\ObSoNoL\mEvs[\treDi])^{-1} \ObSoNoL \treDi \oEvs[\treDi]  \geq  \treDi
\quad\text{ and }\\
(1+1/d)^{-1}(\ObSoNoL\mEvs[\treDi])^{-1} [\gb_\Di-\treRa] \leq (1+1/d)^{-1}(\ObSoNoL\mEvs[\treDi])^{-1} [\gb_\Di-\gb_{\treDi}]  \leq
r_{\Di}.
\end{multline}
Moreover, we set $t_{\Di}:=1\geq
\max_{\Di<j\leq\treDi}\beta_j^2$ and $\mu_{\Di} :=  \Ex
S_{\Di}=v_{\Di}+r_{\Di}$. Introduce the event $\cA_{\Di}:=\{S_{\Di}-\mu_{\Di}\geq - (1/4)(v_{\Di}+2r_{\Di})\}$
 and its complement $\cA_{\Di}^c:=\{S_{\Di}-\mu_{\Di}<
 -(1/4)(v_{\Di}+2r_{\Di})\}$. By employing successively Lemma
 \ref{tr:le:te}, \eqref{ad:le:pm:pr:e2} and $\gb_{\treDi}\leq \treRa$ it   follows now from \eqref{ad:le:pm:pr:e1} that
\begin{multline*}
  \Ex_{\TrSo} p_{\RvDi|\ObSo}(\Di) \leq
  \Ex_{\TrSo}\exp\big(\{-(S_{\Di}-\mu_{\Di})-\mu_{\Di}+3C_\Ev[\treDi-\Di]\}/2\big)\Ind{\cA_{\Di}}
  + \Ex_{\TrSo}\Ind{\cA_{\Di}^c}\\
  \leq\exp
  \big(\{-3v_{\Di}/4-r_{\Di}/2+3C_\Ev[\treDi-\Di]\}/2\big)
+\exp
  \big(-(1/64)(v_{\Di} + 2r_{\Di})\big)\\
\leq \exp
  \big(-r_{\Di}/4+3C_\Ev\treDi/2\big)
+\exp
  \big(-r_{\Di}/32\big)\\
\leq \exp
  \big( -\frac{[\gb_\Di-\treRa]}{4(1+1/d)\ObSoNoL\mEvs[\treDi]}+\frac{3C_\Ev L_\Ev\treRa}{2\ObSoNoL\mEvs[\treDi]}\}\big)+\exp
  \big(-\frac{[\gb_\Di-\treRa]}{32(1+1/d)\ObSoNoL\mEvs[\treDi]})\big)\\
\leq \exp
  \big(-\frac{\gb_\Di}{4(1+1/d)\ObSoNoL\mEvs[\treDi]}  + \frac{2C_\Ev
    L_\Ev\treRa}{\ObSoNoL\mEvs[\treDi]}\big)
\times\exp\big(-\frac{L_\Ev C_{\Ev}\treRa}{4\ObSoNoL\mEvs[\treDi]}\big)
\\
+\exp
  \big(-\frac{[\gb_\Di-\treRa]}{32(1+1/d)\ObSoNoL\mEvs[\treDi]}\big)
\end{multline*}
Taking into account the definition \eqref{ad:de:mp} of $\meDi$, i.e.,
$\gb_\Di> 8 L_{\Ev}C_{\Ev}(1+1/d)\treRa$ for all $1\leq\Di<\meDi$, and
$L_\Ev\treRa(\ObSoNoL\mEvs[\treDi])^{-1}\geq \treDi$ due to Assumption \ref{ad:as:ev} \eqref{ad:as:ev:c}, we obtain
\begin{equation*}
  \Ex_{\TrSo} p_{\RvDi|\ObSo}(\Di) \leq \exp\big(-\frac{L_\Ev C_{\Ev}\treRa}{4\ObSoNoL\mEvs[\treDi]}\big)
+\exp
  \big(-\frac{7
    L_{\Ev}C_{\Ev}\treRa}{32\ObSoNoL\mEvs[\treDi]}\big)\leq 2 \exp
  \big(-\frac{7C_{\Ev}}{32}\treDi\big).
\end{equation*}
Thereby, $\Ex_{\TrSo}P_{\RvDi|\ObSo}(1\leq
  \RvDi<\meDi)=\sum_{\Di=1}^{\meDi-1}\Ex_{\TrSo} p_{\RvDi|\ObSo}(\Di)\leq 2 \exp
  \big(-\frac{7C_{\Ev}}{32}\treDi + \log \DiMa \big)$ using that
  $\DiMa\geq \meDi$ which proves the assertion \ref{ad:le:pm:i}.
  Consider now \ref{ad:le:pm:ii}.  The claim holds trivially true
in case $\peDi=\DiMa$, thus suppose $\peDi<\DiMa$ and let
$\DiMa\geq\Di>\peDi\geq \treDi$. Consider again the upper bound given
in \eqref{ad:le:pm:pr:e1} where now
\begin{equation*}
-S_{\Di}
=\sum_{j=\treDi+1}^{\Di}\frac{\Ev_j^2\sigma_j}{\ObSoNoL^2}(Y_j-\Ev_j\RvSoEr_j)^2.
\end{equation*}
Employing the notations $\alpha_j$ and $\beta_j$ introduced in the
proof of \ref{ad:le:pm:i} and keeping in mind Lemma \ref{tr:le:te} we
define $v_{\Di}:=\sum_{j=\treDi+1}^{\Di}\beta_j^2$ and
$r_{\Di}:=\sum_{j=\treDi+1}^{\Di}\alpha_j^2$ where  $1\geq
\beta_j^2\geq (1+1/d)^{-1}$ due to Assumption
\ref{tr:as:pv}. Moreover, from Assumption
\ref{ad:as:ev} \eqref{ad:as:ev:a}  follows
 $\max_{\treDi<j\leq\Di}\Ev_j^2\geq \max_{\treDi<j}\Ev_j^2\leq C_\Ev
 \min_{1\leq j\leq \treDi}\Ev_j^2 =C_\Ev\mEvs[\treDi]^{-1}$  and
 taking into account in addition Assumption
\ref{ad:as:ev} \eqref{ad:as:ev:c}  that
\begin{multline}\label{ad:le:pm:pr:e3}
L_\Ev (\ObSoNoL\mEvs[\treDi])^{-1}\treRa\geq \treDi,\quad v_\Di\leq m-\treDi
\quad\text{ and }\\
C_\Ev(\ObSoNoL\mEvs[\treDi])^{-1}\treRa \geq C_\Ev(\ObSoNoL\mEvs[\treDi])^{-1} [\gb_{\treDi}-\gb_\Di]  \geq
r_{\Di}.
\end{multline}
Moreover, we set $t_{\Di}:=1\geq
\max_{\treDi<j\leq\Di}\beta_j^2$ and $\mu_{\Di}:=C_{\Ev}[\Di-\treDi]+C_{\Ev}(\mEvs[\treDi]\ObSoNoL)^{-1}[\gb_{\treDi}-\gb_\Di] \geq \Ex
S_{\Di}=v_{\Di}+r_{\Di}$. Consider now the event
$\cA_{\Di}:=\{-S_{\Di}-\mu_{\Di}\leq (C_{\Ev}[\Di -
\treDi]+2C_{\Ev}(\mEvs[\treDi]\ObSoNoL)^{-1}[\gb_{\treDi}-\gb_\Di])\}$
and its complement  $\cA_{\Di}^c:=\{-S_{\Di}-\mu_{\Di}> (C_{\Ev}[\Di -
\treDi]+2C_{\Ev}(\mEvs[\treDi]\ObSoNoL)^{-1}[\gb_{\treDi}-\gb_\Di])\}$. By employing successively Lemma
 \ref{tr:le:te}, \eqref{ad:le:pm:pr:e3} and $\gb_{\treDi}\leq \treRa$ it   follows now from \eqref{ad:le:pm:pr:e1} that
\begin{multline*}
\Ex_{\TrSo} p_{\RvDi|\ObSo}(\Di)
 \leq \Ex_{\TrSo}\exp\big(\{(-S_{\Di}-\mu_{\Di})+\mu_{\Di}+3C_\Ev[\treDi-\Di]\}/2\big)\Ind{\cA_{\Di}}
+ \Ex_{\TrSo}\Ind{\cA_{\Di}^c}\\
\leq\exp
\big(\{2C_{\Ev}[\Di - \treDi]+3C_{\Ev}(\mEvs[\treDi]\ObSoNoL)^{-1}[\gb_{\treDi}-\gb_\Di]+3C_\Ev[\treDi-\Di]\}/2\big)\\+\exp
\big(-\{C_{\Ev}[\Di - \treDi]+2C_{\Ev}(\mEvs[\treDi]\ObSoNoL)^{-1}[\gb_{\treDi}-\gb_\Di]\}/9\big)\\
\leq\exp
\big(\{C_{\Ev}[\treDi-\Di]+3C_{\Ev}(\mEvs[\treDi]\ObSoNoL)^{-1}\treRa\}/2\big)+\exp
\big(-C_\Ev [\Di-\treDi]/9 \big)\\
\leq\exp
\big(C_{\Ev}\{- \Di +3(\mEvs[\treDi]\ObSoNoL)^{-1}\treRa +L_\Ev(\ObSoNoL\mEvs[\treDi])^{-1}\treRa\}/2\big)\\+\exp
\big(-C_{\Ev}(\Di -L_\Ev(\ObSoNoL\mEvs[\treDi])^{-1}\treRa)/9\big)\\
\leq\exp
\big(C_{\Ev}\{- \Di
+5L_\Ev(\mEvs[\treDi]\ObSoNoL)^{-1}\treRa\}/2\big)\times \exp
\big(-\frac{C_{\Ev}L_\Ev\treRa}{2\mEvs[\treDi]\ObSoNoL} \big)\\+\exp
\big(-C_{\Ev}(\Di -L_\Ev(\ObSoNoL\mEvs[\treDi])^{-1}\treRa)/9\big)
\end{multline*}
Taking into account the definition \eqref{ad:de:mp} of $\peDi$, i.e.,
$\Di> 5 L_{\Ev}(\ObSoNoL\mEvs[\treDi])^{-1}\treRa$ for all $\DiMa\geq\Di>\peDi$, and
$L_\Ev\treRa(\ObSoNoL\mEvs[\treDi])^{-1}\geq \treDi$ due to Assumption \ref{ad:as:ev} \eqref{ad:as:ev:c}, we obtain
\begin{equation*}
  \Ex_{\TrSo} p_{\RvDi|\ObSo}(\Di) \leq \exp\big(-\frac{L_\Ev C_{\Ev}\treRa}{2\ObSoNoL\mEvs[\treDi]}\big)
+\exp
  \big(-\frac{4
    L_{\Ev}C_{\Ev}\treRa}{9\ObSoNoL\mEvs[\treDi]}\big)\leq 2 \exp
  \big(-\frac{4C_{\Ev}}{9}\treDi\big).
\end{equation*}
Thereby, $\Ex_{\TrSo}P_{\RvDi|\ObSo}(\peDi<\RvDi\leq \DiMa)=\sum_{\Di=\peDi+1}^{\DiMa}\Ex_{\TrSo} p_{\RvDi|\ObSo}(\Di)\leq 2 \exp
  \big(-\frac{4C_{\Ev}}{9}\treDi + \log \DiMa \big)$  which shows the
  assertion \ref{ad:le:pm:ii} and  completes the proof.
\end{proof}
\begin{proof}[\noindent{\color{darkred}\sc Proof of Lemma
    \ref{ad:le:pv}.}]\label{ad:le:pv:pr}
Consider \ref{ad:le:pv:i}. We start the proof with the observation that 
due to Assumption
\ref{ad:as:ev} \eqref{ad:as:ev:c} the condition
\eqref{tr:as:pv:ac} holds true with $L=L_\Ev$ uniformly for all $m\in\Nz$ and
$\ObSoNoL\in(0,1)$, and hence imposing Assumption \ref{tr:as:pv} the conditions of Corollary
\ref{tr:co:pb} are satisfied, which in turn setting $c:=1/(9K)$ with  $K:=((1+d^{-1})\vee
d^{-2}\HnormV{\TrSo-\RvSoEr}^2)L_\Ev$ implies
 for all $1\leq \Di\leq\DiMa$ and
$\ObSoNoL\in(0,\ObSoNoL_{\TrSy})$   that
\begin{align}\label{ad:le:pv:pr:e1}
& \Ex_{\TrSo}P_{\DiRvSo[\Di]|\ObSo}(
\HnormV{\DiRvSo[\Di]-\TrSo}^2> (4+({11}/{2})K)[\gb_{\Di}\vee\ObSoNoL {\Di}
\oEvs[\Di]])\leq 2\exp(-{\Di}/{36});\\\label{ad:le:pv:pr:e2}
& \Ex_{\TrSo}P_{\DiRvSo[\Di]|\ObSo}( \HnormV{\DiRvSo[\Di]-\TrSo}^2<\{9(1+1/d)\}^{-1}[\gb_{\Di}\vee\ObSoNoL {\Di}
\oEvs])\leq 2\exp(-\Di/(162K^2)).\hfill
\end{align}
On the other hand side, taking into account the definition \eqref{ad:de:mp}
of $\peDi$ and $\meDi$, and the monotonicity of  $(\gb_\Di)_{\Di\geq1}$ and $(\ObSoNoL \Di \oEvs[\Di])_{\Di\geq1}$
we have  for all $\meDi\leq \Di\leq \treDi$ that
\[\ObSoNoL \Di \oEvs\leq \ObSoNoL \treDi
\oEvs[\treDi]\leq \treRa\quad\text{and}\quad\gb_{\Di}\leq 8 L_\Ev C_{\Ev}(1+1/d) \treRa \]
while for all $\peDi\geq\Di\geq\treDi$ (keeping in mind Assumption
\ref{ad:as:tr}) hold
\begin{multline*}
  \Di \leq 5 L_\Ev (\ObSoNoL\mEvs[\treDi])^{-1} \treRa\leq 5 L_\Ev
  (\ObSoNoL\mEvs[\treDi])^{-1} (\kappa^o)^{-1}\ObSoNoL \treDi
  \oEvs[\treDi] \leq 
  (5
  L_\Ev/\kappa^{\TrSy}) \treDi\leq D^{\TrSy}\treDi
  \quad\text{and}\\\quad\gb_{\Di}\leq\gb_{\treDi}\leq \treRa
\end{multline*}
where $D^{\TrSy}:=D^{\TrSy}(\RvSoEr,\TrSo,\Ev):=\ceil{5 L_\Ev/\kappa^{\TrSy}}$. Due to
Assumption \ref{ad:as:ev} \eqref{ad:as:ev:b} and \eqref{ad:as:ev:c}
it follows from $\Di\leq  D^{\TrSy}\treDi$ that $\mEvs\leq
\mEvs[D^{\TrSy}\treDi]\leq \mEvs[D^{\TrSy}]\mEvs[\treDi]$ and $\oEvs\leq \mEvs\leq
\mEvs[D^{\TrSy}]\mEvs[\treDi]\leq \mEvs[D^{\TrSy}] L_\Ev \oEvs[\treDi]$ which in
turn implies $\ObSoNoL \Di \oEvs\leq   L_\Ev D^o\mEvs[D^{\TrSy}]\ObSoNoL
\treDi \oEvs[\treDi]\leq L_\Ev D^{\TrSy}\mEvs[D^{\TrSy}]\treRa$ for all $m\leq\peDi$. Combining the
upper bounds we have  $(4+11K/2)[\gb_\Di\vee\ObSoNoL \Di \oEvs]\leq K^{\TrSy}\treRa$  for all  $\meDi\leq\Di\leq\peDi$ since
$K^{\TrSy}\geq(4+11K/2)(8L_\Ev C_\Ev(1+1/d)\vee L_\Ev D^{\TrSy}\mEvs[D^{\TrSy}])$, and
together with \eqref{ad:le:pv:pr:e1} follows
\begin{multline*}
\sum_{\Di=\meDi}^{\peDi}\Ex_{\TrSo} P_{\DiRvSo|\ObSo}\big(
\HnormV{\DiRvSo[\Di]-\TrSo}^2> K^o\treRa\big)\\\leq \sum_{\Di=\meDi}^{\peDi}\Ex_{\TrSo}P_{\DiRvSo|\ObSo}\big(
\HnormV{\DiRvSo[\Di]-\TrSo}^2> (4+({11}/{2})K)[\gb_{\Di}\vee\ObSoNoL
{\Di}\oEvs[\Di]]\big)\\\leq 2\sum_{\Di=\meDi}^{\peDi}\exp(-m/36)\leq 74 \exp(-\meDi/36)
\end{multline*}
which proves the assertion \ref{ad:le:pv:i}. Consider now
\ref{ad:le:pv:ii}. We observe that by definition \eqref{tr:de:tre} of
$\treRa$ for all $\Di\in\Nz$ holds $\treRa\leq [\ObSoNoL \Di
\oEvs\vee\gb_\Di]$, and hence
$\{9(1+1/d)\}^{-1}[\gb_{\Di}\vee\ObSoNoL \Di \oEvs]\geq
(K^{\TrSy})^{-1}\treRa$ since $K^{\TrSy}\geq 9(1+1/d)$. Combining the last
estimate, \eqref{ad:le:pv:pr:e2} and $K^{\TrSy}\geq 10 K$ it follows that
\begin{multline*}
\sum_{\Di=\meDi}^{\peDi}\Ex_{\TrSo}
P_{\DiRvSo|\ObSo}\big(
\HnormV{\DiRvSo[\Di]-\TrSo}^2<(K^{\TrSy})^{-1}\treRa\big)\\\leq
\sum_{\Di=\meDi}^{\peDi}\Ex_{\TrSo}
P_{\DiRvSo|\ObSo}\big(
\HnormV{\DiRvSo[\Di]-\TrSo}^2<\{9(1+1/d)\}^{-1}[\gb_{\Di}\vee\ObSoNoL \Di \oEvs]
\big)\\\leq
2\sum_{\Di=\meDi}^{\peDi}\exp(-m/(K^{\TrSy})^2)\leq 4(K^{\TrSy})^2\exp(-\meDi/(K^{\TrSy})^2)
\end{multline*}
which shows the
  assertion \ref{ad:le:pv:ii} and  completes the proof.
\end{proof}
\begin{proof}[\noindent{\color{darkred}\sc Proof of Theorem
    \ref{ad:th:ora}.}]\label{ad:th:ora:pr}
We start the proof with the observation that  Lemma \ref{ad:le:pm}
together with Lemma \ref{ad:le:pv} \ref{ad:le:pv:i} imply
\begin{multline}\label{ad:th:ora:pr:e1}
\Ex_{\TrSo}P_{\RvDiRvSo|\ObSo}(
\HnormV{\RvDiRvSo-\TrSo}^2> K^o\treRa)=\Ex_{\TrSo}\sum_{\Di=1}^{\DiMa}p_{\RvDi|\ObSo}(\Di)P_{\DiRvSo|\ObSo}(
\HnormV{\DiRvSo[\Di]-\TrSo}^2> K^o\treRa)\\
\leq \Ex_{\TrSo}P_{\RvDi|\ObSo}(1\leq
  \RvDi<\meDi)+\Ex_{\TrSo}P_{\RvDi|\ObSo}(\peDi<
  \RvDi\leq\DiMa)\\\hfill + \sum_{\Di=\meDi}^{\peDi}\Ex_{\TrSo}P_{\DiRvSo|\ObSo}(
\HnormV{\DiRvSo[\Di]-\TrSo}^2> K^o\treRa)\\
\leq 4\exp\big(-\treDi\{C_{\Ev}/5-\log\DiMa/\treDi\}\big) + 74
\exp(-\meDi/36)
\end{multline}
 On the other hand side, from Lemma \ref{ad:le:pm}
together with Lemma \ref{ad:le:pv} \ref{ad:le:pv:ii} also follows that
\begin{multline}\label{ad:th:ora:pr:e2}
\Ex_{\TrSo}P_{\RvDiRvSo|\ObSo}(
\HnormV{\RvDiRvSo-\TrSo}^2< (K^o)^{-1}\treRa)
\leq \Ex_{\TrSo}P_{\RvDi|\ObSo}(1\leq
  \RvDi<\meDi)\\\hfill+\Ex_{\TrSo}P_{\RvDi|\ObSo}(\peDi<
  \RvDi\leq\DiMa) + \Ex_{\TrSo}\sum_{\Di=\meDi}^{\peDi}p_{\RvDi|\ObSo}(\Di)P_{\DiRvSo|\ObSo}(
\HnormV{\DiRvSo[\Di]-\TrSo}^2< (K^o)^{-1}\treRa)\\
\leq 4\exp\big(-\treDi\{C_{\Ev}/5-\log\DiMa/\treDi\}\big) + 4(K^{\TrSy})^2
\exp(-\meDi/(K^{\TrSy})^2)
\end{multline}
By combining \eqref{ad:th:ora:pr:e1} and \eqref{ad:th:ora:pr:e2} we
obtain the assertion of the theorem since $\meDi,\treDi \to \infty $ and
$\log\DiMa/\treDi=o(1)$ as $\ObSoNoL\to0$ which completes the proof.
\end{proof}
\subsection{Proof of Theorem \ref{ad:th:bo}}
The next assertion presents a concentration inequality for
Gaussian random variables. 
\begin{lem}\label{ad:le:te}Let the assumptions of Lemma \ref{tr:le:te}
  be satisfied. For all $c\geq 0$ we have
  \begin{align}\label{ad:le:te:e1}
   &\sup_{m\geq 1}(6t_m)^{-1}\exp
    \bigg(\frac{c(v_m+2r_m)}{4t_m}\bigg)\Ex\bigg(S_m-\Ex S_m-\frac{3}{2}c(v_m+2r_m)\bigg)_+\leq1
  \end{align}
where $(a)_+:=(a\vee0)$.\end{lem}
\begin{proof}[\noindent{\color{darkred}\sc Proof of Lemma \ref{ad:le:te}.}]\label{ad:le:te:pr}
The assertion follows from Lemma \ref{tr:le:te} (keeping in mind that
$c\geq 1$), indeed
\begin{multline*}
\Ex\bigg(S_m-\Ex S_m-\frac{3}{2}c(v_m+2r_m)\bigg)_+=\int_{0}^\infty P(S_m-\Ex S_m \geq x+ \frac{3}{2}c(v_m+2r_m))dx\\ =
\int_{0}^\infty P(S_m-\Ex S_m \geq \frac{3}{2}(2x/(3(v_m+2r_m))+ c)(v_m+2r_m))dx\\
\leq\int_{0}^\infty \exp \bigg(-\frac{(2x/(3(v_m+2r_m))+
  c)(v_m+2r_m)}{4t_m}\bigg)dx\\=\int_{0}^\infty \exp \bigg(-\frac{2x/3+
  c(v_m+2r_m)}{4t_m}\bigg)dx\\
=\exp \bigg(-\frac{c(v_m+2r_m)}{4t_m}\bigg)\int_{0}^\infty \exp \bigg(-\frac{x}{6t_m}\bigg)dx=\exp \bigg(-\frac{c(v_m+2r_m)}{4t_m}\bigg)(6t_m)
\end{multline*}
\end{proof}
\begin{lem}\label{ad:le:bo:rem}
If Assumption \ref{tr:as:pv} and \ref{ad:as:ev} hold true then for all $\ObSoNoL\in(0,\ObSoNoL_{\TrSy})$
\begin{enumerate}[label=\emph{\textbf{(\roman*)}},ref=\emph{\textbf{(\roman*)}}]
\item\label{ad:le:bo:rem:i} $\sum_{j=1}^{\DiMa}\sigma_j^2\Ev_j^2\ObSoNoL^{-2}\Ex_{\TrSo}\{(\ObSo_j-\Ev_j\TrSo[j])P_{\RvDi|\ObSo}(j\leq
\RvDi\leq\DiMa)\}^2$\\
\null\hfill
$\leq\ObSoNoL \peDi \oEvs[\peDi]+
10\Evs_1\exp\big(-\treDi/5+2\log\DiMa\big)$;
\item\label{ad:le:bo:rem:ii}
  $\sum_{j=1}^{\DiMa}(\RvSoEr_j-\TrSo[j])^2\Ex_{\TrSo}\Ex_{\RvDi|\ObSo}\{\Ind{\{1\leq\RvDi<j\}}
  + (\sigma_j/\RvSoVa_j)^2\Ind{\{j\leq\RvDi\leq \DiMa\}}\}+
  \sum_{j>\DiMa}(\RvSoEr_j-\TrSo[j])^2$\\\null\hfill$\leq
  \gb_{\meDi} +\HnormV{\RvSoEr-\TrSo}^2\{d^{-2}\ObSoNoL\mEvs[\meDi]+ 2\exp\big(-\treDi/5+\log\DiMa\big)\}$.
\end{enumerate}
\end{lem}
\begin{proof}[\noindent{\color{darkred}\sc Proof of Lemma \ref{ad:le:bo:rem}.}]\label{ad:le:bo:rem:pr}
Consider \ref{ad:le:bo:rem:i}. We start with the observation that the random variables
$\{\ObSoNo_j:=\ObSoNoL^{-1/2}(\ObSo_j-\Ev_j\TrSo[j])\}_{j\geq1}$ are
independent and standard normally distributed. Moreover, applying
Jensen's inequality we have
\begin{multline*}
  \{\ObSoNo_jP_{\RvDi|\ObSo}(j\leq
  \RvDi\leq\DiMa)\}^2=\{\Ex_{\RvDi|\ObSo}\ObSoNo_j\Ind{\{j\leq\RvDi\leq\DiMa\}}\}^2\leq \Ex_{\RvDi|\ObSo}\ObSoNo_j^2\Ind{\{j\leq\RvDi\leq\DiMa\}}
  \\
  =\ObSoNo_j^2P_{\RvDi|\ObSo}(j\leq
  \RvDi\leq\DiMa).
\end{multline*}
We split the sum into two parts which we bound separately. Precisely,
\begin{multline}\label{ad:le:bo:rem:pr:e1}
\sum_{j=1}^{\DiMa}\sigma_j^2\Ev_j^2\ObSoNoL^{-2}\{(\ObSo_j-\Ev_j\TrSo[j])P_{\RvDi|\ObSo}(j\leq
\RvDi\leq\DiMa)\}^2\\
\leq \sum_{j=1}^{\peDi}\ObSoNoL\Evs_j\ObSoNo^2_j +
\sum_{j=1}^{\DiMa}\ObSoNoL\Evs_j\ObSoNo_j^2P_{\RvDi|\ObSo}(\peDi<
\RvDi\leq\DiMa)
\end{multline}
where we used that $\sigma_j\leq \ObSoNoL\Evs_j$.
Keeping in mind the notations used in Lemma \ref{ad:le:te} let
$S_{\DiMa}:=\sum_{j=1}^{\DiMa}\ObSoNoL\Evs_j\ObSoNo_j^2$ and
observe that $\alpha_j=0$ and $\beta_j^2=\ObSoNoL\Evs_j$, and hence
$r_{\DiMa}=0$.  Keeping in mind that
$\DiMa:=\max\{1\leq m\leq \gauss{\ObSoNoL^{-1}}:
\ObSoNoL\mEvs\leq \Evs_1\}$ we set $t_{\DiMa}:=\Evs_1\geq \ObSoNoL \mEvs[\DiMa]
=\max_{1\leq j\leq\DiMa}\beta_j^2$  and $v_{\DiMa} := \Evs_1\DiMa =
\DiMa t_{\DiMa}\geq\sum_{j=1}^{\DiMa}\beta_j^2$, where $\Ex_{\TrSo} S_{\DiMa}\leq
v_{\DiMa}$. From Lemma \ref{ad:le:te} with $c=2/3$ follows that $
\Ex_{\TrSo}(S_{\DiMa}-2\Evs_1\DiMa)_+\leq
(6t_{\DiMa})\exp(-v_{\DiMa}/(6t_{\DiMa}))=(6\Evs_1)\exp(-\DiMa/6)$,
and hence
\begin{multline}\label{ad:le:bo:rem:pr:e2}
\sum_{j=1}^{\DiMa}\ObSoNoL\Evs_j\Ex_{\TrSo}\{\ObSoNo_j^2 P_{\RvDi|\ObSo}(\peDi<\RvDi\leq\DiMa)\}
\\\leq
\Ex\big(S_{\DiMa}-2\Evs_1\DiMa\big)_++2\Evs_1\DiMa\Ex_{\TrSo}P_{\RvDi|\ObSo}(\peDi<
\RvDi\leq\DiMa)\\
\leq 6\Evs_1\exp(-\DiMa/6)+2\Evs_1\DiMa\Ex_{\TrSo}P_{\RvDi|\ObSo}(\peDi<\RvDi\leq\DiMa).
\end{multline}
We distinguish two cases. First, if $\peDi=\DiMa$, then assertion \ref{ad:le:bo:rem:i} follows  by combining
\eqref{ad:le:bo:rem:pr:e1} and $\Ex_{\TrSo}\ObSoNo_j^2=1$. Second, if
$\peDi<\DiMa$, then the definition \eqref{ad:de:mp} of $\peDi$ implies
$\DiMa>5\treDi$ which in turn implies the assertion \ref{ad:le:bo:rem:i} by combining
\eqref{ad:le:bo:rem:pr:e1}, $\Ex_{\TrSo}\ObSoNo_j^2=1$,  \eqref{ad:le:bo:rem:pr:e2} and Lemma \ref{ad:le:pm} \ref{ad:le:pm:ii}. Consider
\ref{ad:le:bo:rem:ii}. Due to Assumption \ref{tr:as:pv} we have
$(\sigma_j/\RvSoVa_j)^{2}\leq (1\wedge
d^{-2}\ObSoNoL\Evs_j)$ which we will use without further reference. Splitting the first sum into two parts we
obtain
\begin{multline*}
\sum_{j=1}^{\DiMa}(\RvSoEr_j-\TrSo[j])^2\Ex_{\TrSo}\{\Ind{\{1\leq\RvDi<j\}}
  + (\sigma_j/\RvSoVa_j)^2\Ind{\{j\leq\RvDi\leq \DiMa\}}\}+ \sum_{j>\DiMa}(\RvSoEr_j-\TrSo[j])^2\\
\hspace*{5ex}\leq  \sum_{j=1}^{\meDi}(\RvSoEr_j-\TrSo[j])^2\Ex_{\TrSo}\{\Ind{\{1\leq\RvDi<j\}}
  + d^{-2}\ObSoNoL\Evs_j\}\hfill\\
+ \sum_{j=\meDi+1}^{\DiMa}(\RvSoEr_j-\TrSo[j])^2\Ex_{\TrSo}\{\Ind{\{1\leq\RvDi<j\}}
  + \Ind{\{j\leq\RvDi\leq \DiMa\}}\}+
  \sum_{j>\DiMa}(\RvSoEr_j-\TrSo[j])^2\\\hfill
\leq \HnormV{\RvSoEr-\TrSo}^2\{\Ex_{\TrSo}P_{\RvDi|\ObSo}(1\leq\RvDi<\meDi)+
d^{-2}\ObSoNoL\mEvs[\meDi]\}+\sum_{j>\meDi}(\RvSoEr_j-\TrSo[j])^2
\end{multline*}
The assertion \ref{ad:le:bo:rem:ii} follows now by combining
the last estimate and Lemma \ref{ad:le:pm} \ref{ad:le:pm:i}, which
completes the proof.
\end{proof}
\begin{proof}[\noindent{\color{darkred}\sc Proof of Theorem
    \ref{ad:th:bo}.}]\label{ad:th:bo:pr}
We start the proof with the observation that
$\hSo_j-\TrSo[j]=(\DiPoEr{}{j}-\TrSo[j])P_{\RvDi|\ObSo}(j\leq \RvDi\leq
  \DiMa)+(\RvSoEr_j-\TrSo[j])P_{\RvDi|\ObSo}(1\leq\RvDi<j)\}$ for all $1\leq
  j\leq\DiMa$ and $\hSo_j-\TrSo[j]=\RvSoEr_j-\TrSo[j]$ for all $j>\DiMa$. From the identity
  $\DiPoEr{}{j}-\TrSo[j]=(\sigma_j/\RvSoVa_j)(\RvSoEr_j-\TrSo[j])+(\sigma_j\Ev_j\ObSoNoL^{-1})(\ObSo_j-\Ev_j\TrSo[j])$
  and Lemma \ref{ad:le:bo:rem}  follows that
  \begin{multline*}
   \Ex_{\TrSo} \HnormV{\hSo-\TrSo}^2\leq \sum_{j=1}^{\DiMa}2\sigma^2\Ev^2_j\ObSoNoL^{-2}\Ex_{\TrSo}\{(\ObSo_j-\Ev_j\TrSo[j])P_{\RvDi|\ObSo}(j\leq \RvDi\leq
  \DiMa)\\ + \sum_{j=1}^{\DiMa}2(\RvSoEr_j-\TrSo[j])^2\Ex_{\TrSo}\{(\sigma_j/\RvSoVa_j)P_{\RvDi|\ObSo}(j\leq \RvDi\leq
  \DiMa)+
  P_{\RvDi|\ObSo}(1\leq\RvDi<j)\}^2+\sum_{j>\DiMa}(\RvSoEr_j-\TrSo[j])^2\\
\leq 2\{\ObSoNoL \peDi \oEvs[\peDi]
+10\Evs_1\exp\big(-\treDi/5+2\log\DiMa\big)\}\\
+ 2\{\gb_{\meDi} +\HnormV{\RvSoEr-\TrSo}^2\{d^{-2}\ObSoNoL\mEvs[\meDi]+ 2\exp\big(-\treDi/5+\log\DiMa\big)\}\}.
 \end{multline*}
On the other hand side, taking into account the definition \eqref{ad:de:mp} of $\meDi$ and
$\peDi$, we have show in the proof of Lemma \ref{ad:le:pv} that $\gb_{\meDi}\leq 8L_\Ev C_{\Ev}(1+1/d)\treRa$   and
$\ObSoNoL \peDi \oEvs[\peDi]\leq L_\Ev D^o\mEvs[D^o]\treRa$, while
trivially
$\ObSoNoL\mEvs[\meDi]\leq\ObSoNoL\mEvs[\treDi]\leq\treRa$. By
combination of these estimates we obtain
  \begin{multline*}
   \Ex_{\TrSo} \HnormV{\hSo-\TrSo}^2\leq \{2L_\Ev
   D^o\mEvs[D^o]+16 L_\Ev C_{\Ev}(1+1/d)+2
   d^{-2}\HnormV{\RvSoEr-\TrSo}^2\}\treRa\\
+(20\Evs_1+ 4\HnormV{\RvSoEr-\TrSo}^2)\exp\big(-\treDi/5+2\log\DiMa-\log\treRa\big)\}\treRa
 \end{multline*}
From the last bound follows the assertion of the theorem since
$(2\log\DiMa-\log\treRa)/\treDi \to 0$ as $\ObSoNoL\to0$ which completes the proof.
\end{proof}

\subsection{Proof of Theorem \ref{ad:th:ra:mm}}
\begin{proof}[\noindent{\color{darkred}\sc Proof of Lemma
    \ref{ad:le:opm}.}]\label{ad:le:opm:pr} The proof follows along the
  lines of the proof of Lemma \ref{ad:le:pm}, where we replace
  $\meDi$, $\peDi$, $\treDi$ and $\treRa$ by its counterpart
  $\moeDi$, $\poeDi$, $\oeDi$ and $\oeRa$, respectively.
 Moreover, we will use without  further reference, that for all
 $\TrSo\in\cwrSo$  the bias is bound by  $\gb_\Di\leq
 \rSo \wSo_m$, for all $m\in\Nz$, and hence $\gb_{\oeDi}\leq (1\vee\rSo)\oeRa$.

Consider \ref{ad:le:opm:i}. The claim holds trivially true
in case $\moeDi=1$, thus suppose $\moeDi>1$ and let
$1\leq \Di<\moeDi\leq \oeDi$. Define
$S_{\Di}:=\normV{\hSo^{\oeDi}-\RvSoEr}_{\sigma}^2-\normV{\hSo^\Di-\RvSoEr}_{\sigma}^2$. Let
 $\cA_{\Di}$ and $\cA_{\Di}^c$,
 respectively, be an event and its complement defined as in the Proof of Lemma \ref{ad:le:pm}, then it follows
\begin{equation}\label{ad:le:opm:pr:e1}
 p_{\RvDi|\ObSo}(\Di)\leq \exp\bigg(\frac{1}{2}\big\{-S_{\Di}+3 C_\Ev [\oeDi-\Di]\big\}\bigg)\Ind{\cA_{\Di}}
+ \Ind{\cA_{\Di}^c}
\end{equation}
where $S_{\Di}=\sum_{j=\Di+1}^{\oeDi}\frac{\Ev_j^2\sigma_j}{\ObSoNoL^2}(Y_j-\Ev_j\RvSoEr_j)^2$.
We use the
notation introduced in Lemma \ref{ad:le:pm}, where again $1\geq
\beta_j^2\geq (1+1/d)^{-1}$  due  to Assumption
\ref{tr:as:pv}  and by
employing $\min_{\Di<j\leq\oeDi}\Ev_j^2\geq \mEvs[\oeDi]^{-1}$
together with Assumption
\ref{ad:as:ev} \eqref{ad:as:ev:c}
\begin{multline}\label{ad:le:opm:pr:e2}
L_\Ev (\ObSoNoL\mEvs[\oeDi])^{-1}(1\vee\rSo)\oeRa\geq L_\Ev (\ObSoNoL\mEvs[\oeDi])^{-1} \ObSoNoL \oeDi \oEvs[\oeDi]  \geq  \oeDi
\quad\text{ and }\\
(1+1/d)^{-1}(\ObSoNoL\mEvs[\oeDi])^{-1} [\gb_\Di-(1\vee\rSo)\oeRa] \leq (1+1/d)^{-1}(\ObSoNoL\mEvs[\oeDi])^{-1} [\gb_\Di-\gb_{\oeDi}]  \leq
r_{\Di}.
\end{multline}
By employing successively Lemma
 \ref{tr:le:te}, \eqref{ad:le:opm:pr:e2} and $\gb_{\oeDi}\leq
 (1\vee\rSo)\oeRa$ for all $\TrSo\in\cwrSo$ it   follows now from \eqref{ad:le:opm:pr:e1} that
\begin{multline*}
  \Ex_{\TrSo} p_{\RvDi|\ObSo}(\Di) 
\leq \exp
  \big(-\frac{\gb_\Di}{4(1+1/d)\ObSoNoL\mEvs[\oeDi]}  + \frac{2C_\Ev
    L_\Ev(1\vee\rSo)\oeRa}{\ObSoNoL\mEvs[\oeDi]}\big)
\times\exp\big(-\frac{L_\Ev C_{\Ev}(1\vee\rSo)\oeRa}{4\ObSoNoL\mEvs[\oeDi]}\big)
\\
+\exp
  \big(-\frac{[\gb_\Di-(1\vee\rSo)\oeRa]}{32(1+1/d)\ObSoNoL\mEvs[\oeDi]}\big).
\end{multline*}
Taking into account the definition \eqref{ad:de:omp} of $\moeDi$, i.e.,
$\gb_\Di> 8 L_{\Ev}C_{\Ev}(1+1/d)(1\vee\rSo)\oeRa$ for all $1\leq\Di<\moeDi$, and
$L_\Ev\oeRa(\ObSoNoL\mEvs[\oeDi])^{-1}\geq \oeDi$ due to Assumption \ref{ad:as:ev} \eqref{ad:as:ev:c}, we obtain
\begin{equation*}
  \Ex_{\TrSo} p_{\RvDi|\ObSo}(\Di) \leq \exp\big(-\frac{L_\Ev C_{\Ev}(1\vee\rSo)\oeRa}{4\ObSoNoL\mEvs[\oeDi]}\big)
+\exp
  \big(-\frac{7
    L_{\Ev}C_{\Ev}(1\vee\rSo)\oeRa}{32\ObSoNoL\mEvs[\oeDi]}\big)\leq 2 \exp
  \big(-\frac{7C_{\Ev}(1\vee\rSo)}{32}\oeDi\big).
\end{equation*}
Thereby, $\Ex_{\TrSo}P_{\RvDi|\ObSo}(1\leq
  \RvDi<\moeDi)=\sum_{\Di=1}^{\moeDi-1}\Ex_{\TrSo} p_{\RvDi|\ObSo}(\Di)\leq 2 \exp
  \big(-\frac{7C_{\Ev}(1\vee\rSo)}{32}\oeDi + \log \DiMa \big)$ using that
  $\DiMa\geq \moeDi$ which proves the assertion \ref{ad:le:opm:i}.
  Consider now \ref{ad:le:opm:ii}.  The claim holds trivially true
in case $\poeDi=\DiMa$, thus suppose $\poeDi<\DiMa$ and let
$\DiMa\geq\Di>\poeDi\geq \oeDi$. Consider the upper bound \eqref{ad:le:opm:pr:e1} where $-S_{\Di}=\sum_{j=\oeDi+1}^{\Di}\frac{\Ev_j^2\sigma_j}{\ObSoNoL^2}(Y_j-\Ev_j\RvSoEr_j)^2$.
Employing the notations introduced in the
Proof of Lemma \ref{ad:le:pm} where we had $1\geq
\beta_j^2\geq (1+1/d)^{-1}$ due to Assumption
\ref{tr:as:pv}, we obtain from Assumption
\ref{ad:as:ev} \eqref{ad:as:ev:a}  that
$\max_{\oeDi<j\leq\Di}\Ev_j^2\leq C_\Ev\mEvs[\oeDi]^{-1}$ 
and
 taking into account in addition Assumption
\ref{ad:as:ev} \eqref{ad:as:ev:c}  that
\begin{multline}\label{ad:le:opm:pr:e3}
L_\Ev (\ObSoNoL\mEvs[\oeDi])^{-1}(1\vee\rSo)\oeRa\geq \oeDi,\quad v_\Di\leq m-\oeDi
\quad\text{ and }\\
C_\Ev(\ObSoNoL\mEvs[\oeDi])^{-1}(1\vee\rSo)\oeRa \geq C_\Ev(\ObSoNoL\mEvs[\oeDi])^{-1} [\gb_{\oeDi}-\gb_\Di]  \geq
r_{\Di}.
\end{multline}
By employing successively Lemma
 \ref{tr:le:te}, \eqref{ad:le:opm:pr:e3} and $\gb_{\oeDi}\leq (1\vee\rSo)\treRa$ it   follows now from \eqref{ad:le:opm:pr:e1} that
\begin{multline*}
\Ex_{\TrSo} p_{\RvDi|\ObSo}(\Di)
\leq\exp
\big(C_{\Ev}\{- \Di
+5L_\Ev(\mEvs[\oeDi]\ObSoNoL)^{-1}(1\vee\rSo)\oeRa\}/2\big)\times \exp
\big(-\frac{C_{\Ev}L_\Ev(1\vee\rSo)\oeRa}{2\mEvs[\oeDi]\ObSoNoL} \big)\\+\exp
\big(-C_{\Ev}(\Di -L_\Ev(\ObSoNoL\mEvs[\oeDi])^{-1}(1\vee\rSo)\oeRa)/9\big)
\end{multline*}
Taking into account the definition \eqref{ad:de:omp} of $\poeDi$, i.e.,
$\Di> 5 L_{\Ev}(\ObSoNoL\mEvs[\oeDi])^{-1}(1\vee\rSo)\oeRa$ for all $\DiMa\geq\Di>\poeDi$, and
$L_\Ev(1\vee\rSo)\oeRa(\ObSoNoL\mEvs[\oeDi])^{-1}\geq (1\vee\rSo)\oeDi$ due to Assumption \ref{ad:as:ev} \eqref{ad:as:ev:c}, we obtain
\begin{equation*}
  \Ex_{\TrSo} p_{\RvDi|\ObSo}(\Di) \leq \exp\big(-\frac{L_\Ev C_{\Ev}(1\vee\rSo)\oeRa}{2\ObSoNoL\mEvs[\oeDi]}\big)
+\exp
  \big(-\frac{4
    L_{\Ev}C_{\Ev}(1\vee\rSo)\oeRa}{9\ObSoNoL\mEvs[\oeDi]}\big)\leq 2 \exp
  \big(-\frac{4C_{\Ev}(1\vee\rSo)}{9}\oeDi\big).
\end{equation*}
Thereby, $\Ex_{\TrSo}P_{\RvDi|\ObSo}(\poeDi<\RvDi\leq \DiMa)=\sum_{\Di=\poeDi+1}^{\DiMa}\Ex_{\TrSo} p_{\RvDi|\ObSo}(\Di)\leq 2 \exp
  \big(-\frac{4C_{\Ev}(1\vee\rSo)}{9}\oeDi + \log \DiMa \big)$  which shows the
  assertion \ref{ad:le:opm:ii} and  completes the proof.
\end{proof}
\begin{proof}[\noindent{\color{darkred}\sc Proof of Theorem
    \ref{ad:th:ra:mm}.}]\label{ad:th:ra:mm:pr}
We start the proof with the observation that  
due to Assumption
\ref{ad:as:ev} \eqref{ad:as:ev:c} the condition
\eqref{tr:as:pv:ac} holds true with $L=L_\Ev$ uniformly for all $m\in\Nz$ and
$\ObSoNoL\in(0,1)$, and hence imposing Assumption \ref{tr:as:pv} the conditions of Corollary
\ref{tr:co:pb} \eqref{tr:co:pb:e1} are satisfied, which in turn implies, by 
setting   $K:=((1+1/d)\vee
\rSo/d^{2})L_\Ev\geq ((1+d^{-1})\vee
d^{-2}\HnormV{\TrSo-\RvSoEr}^2)L_\Ev$, 
that for all $1\leq \Di\leq\DiMa$ and
$\ObSoNoL\in(0,\ObSoNoL_{\OpSy})$   
 \begin{align}\label{ad:th:ra:mm:pr:e1}
& \Ex_{\TrSo}P_{\DiRvSo[\Di]|\ObSo}(
\HnormV{\DiRvSo[\Di]-\TrSo}^2> (4+(11/{2})K)[\gb_{\Di}\vee\ObSoNoL {\Di}
\oEvs[\Di]])\leq 2\exp(-{\Di}/{36}).\hfill
\end{align}
Moreover, exploiting the inequality below \eqref{tr:pr:pbm:pr:e2} with
  $c_1=1/3$ and $c_2\geq 1$, it is possible to prove a slightly modified version of Corollary  \ref{tr:co:pb} \eqref{tr:co:pb:e1} which implies for all $c_2\geq 1$
 \begin{equation}\label{ad:th:ra:mm:pr:e2}
\Ex_{\TrSo}P_{\DiRvSo[\Di]|\ObSo}(
\HnormV{\DiRvSo[\Di]-\TrSo}^2> 16c_2K[\gb_{\Di}\vee\ObSoNoL {\Di}
\oEvs[\Di]])
\leq 2\exp(-c_2{\Di}/{12}).
\end{equation}%
Consider \ref{ad:th:ra:mm:i}. Following line by line the proof of
Lemma \ref{ad:le:pv} \ref{ad:le:pv:i}, using
\eqref{ad:th:ra:mm:pr:e1} rather than \eqref{ad:le:pv:pr:e1} and
exploiting $[\gb_\Di\vee\ObSoNoL\Di\oEvs]\leq 8 L_\Ev
C_\Ev(1+1/d){(1\vee\rSo)}\oeRa$ for all $\moeDi\leq \Di\leq\oeDi$ and
$[\gb_\Di\vee\ObSoNoL\Di\oEvs]\leq L_\Ev D^{\OpSy}\mEvs[D^{\OpSy}]{(1\vee\rSo)}\oeRa$ with $D^{\OpSy}:=\ceil{5 L_\Ev/\kappa^{\OpSy}}$
for all $\oeDi\leq \Di\leq\poeDi$ (keep in mind that $m\leq D^{\OpSy}\oeDi$) , we obtain
\begin{multline*}
\sum_{\Di=\moeDi}^{\poeDi}\Ex_{\TrSo} P_{\DiRvSo|\ObSo}\big(
\HnormV{\DiRvSo[\Di]-\TrSo}^2> K^{\OpSy}\oeRa\big)\\\leq \sum_{\Di=\moeDi}^{\poeDi}\Ex_{\TrSo}P_{\DiRvSo|\ObSo}\big(
\HnormV{\DiRvSo[\Di]-\TrSo}^2> (4+({11}/{2})K)[\gb_{\Di}\vee\ObSoNoL
{\Di}\oEvs[\Di]]\big)\\\leq 2\sum_{\Di=\moeDi}^{\poeDi}\exp(-m/36)\leq 74 \exp(-\moeDi/36).
\end{multline*}
Combining the last estimate, Lemma \ref{ad:le:opm} and the decomposition
\eqref{ad:th:ora:pr:e1} used in the proof of Theorem
\ref{ad:th:ora} (with $\meDi$ and $\peDi$ replaced by 
  $\moeDi$, $\poeDi$) it follows that
\begin{multline}\label{ad:th:ra:mm:pr:e3}
\Ex_{\TrSo}P_{\RvDiRvSo|\ObSo}(
\HnormV{\RvDiRvSo-\TrSo}^2> K^{\OpSy}\oeRa)
\\
\leq 4\exp\big(-\oeDi\{C_{\Ev}/5-\log\DiMa/\oeDi\}\big) + 74
\exp(-\moeDi/36)
\end{multline}
Taking into account that $\oeDi\to\infty$ and $\log \DiMa/\oeDi=o(1)$
as $\ObSoNoL\to0$, we obtain the assertion \ref{ad:th:ra:mm:i} of the
Theorem for any $\TrSo\in\cwrSo$ such that $\moeDi\to\infty$ as
$\ObSoNoL\to0$. On the other hand side, if   $\TrSo\in\cwrSo$ such
that $\moeDi\not\to\infty$, i.e., $\sup_{\ObSoNoL}\moeDi<\infty$, then there exists
$\ObSoNoL_\TrSy\in(0,1)$ such that $\moeDi[G_{\ObSoNoL_{\TrSy}}]=\moeDi$ for all
$\ObSoNoL\in(0,\ObSoNoL_\TrSy)$ (keep in mind that
$(\moeDi)_{\ObSoNoL}$ is an integer-valued monotonically increasing
sequence). Moreover, by construction
$\gb_{\moeDi[G_{\ObSoNoL_{\TrSy}}]}\leq 8 L_\Ev C_\Ev(1+1/d){(1\vee\rSo)}\oeRa$ for
all $\ObSoNoL\in(0,\ObSoNoL_\TrSy)$ which in turn implies
$\gb_\Di\leq \gb_{\moeDi[G_{\ObSoNoL_{\TrSy}}]}=0$ for all $\Di\geq \moeDi[G_{\ObSoNoL_{\TrSy}}]$, since $\oeRa=o(1)$ as
$\ObSoNoL\to0$. Thereby, for all $\Di\geq \moeDi[G_{\ObSoNoL_{\TrSy}}]$
follows $\oeRa/[\gb_{\Di}\vee\ObSoNoL {\Di}
\oEvs[\Di]]=\oeRa/[\ObSoNoL {\Di}\oEvs[\Di]]\geq[\ObSoNoL
{\oeDi}\oEvs[\oeDi]]/[\ObSoNoL {\Di}\oEvs[\Di]]\geq \oeDi/[L_\Ev
\Di]$ using that $L_\Ev\oEvs[\oeDi]\geq\mEvs[\oeDi]\geq\mEvs\geq\oEvs$ due
to Assumption \ref{ad:as:ev} \eqref{ad:as:ev:c}, which in turn  together
with $K^{\OpSy}\oeRa/[\gb_{\Di}\vee\ObSoNoL {\Di}
\oEvs[\Di]]\geq K^{\OpSy}\oeDi/[L_\Ev {\Di}] = 16 c_2 K$, $c_2:=(8C_\Ev(1+1/d)\vee
D^{\OpSy}\mEvs[D^{\OpSy}]){(1\vee\rSo)} \oeDi/{\Di}\geq1$  and \eqref{ad:th:ra:mm:pr:e2}  implies 
\begin{multline*}
\sum_{\Di=\moeDi}^{\poeDi}\Ex_{\TrSo} P_{\DiRvSo|\ObSo}\big(
\HnormV{\DiRvSo[\Di]-\TrSo}^2> K^{\OpSy}\oeRa\big)
\\\leq 2\exp(-(8C_\Ev(1+1/d)\vee
D^{\OpSy}\mEvs[D^{\OpSy}]){(1\vee\rSo)} \oeDi/12 + \log \DiMa)
\\\leq 2 \exp(-C_\Ev\oeDi/5 + \log \DiMa).
\end{multline*}
Consequently, we have  
\begin{equation*}\Ex_{\TrSo}P_{\RvDiRvSo|\ObSo}(
\HnormV{\RvDiRvSo-\TrSo}^2> K^{\OpSy}\oeRa)
\leq 6\exp\big(-\oeDi\{C_{\Ev}/5-\log\DiMa/\oeDi\}\big) 
\end{equation*}
which shows  that assertion \ref{ad:th:ra:mm:i} holds  for any
$\TrSo\in\cwrSo$ since $\oeDi\to\infty$ and $\log \DiMa/\oeDi=o(1)$
as $\ObSoNoL\to0$. Consider \ref{ad:th:ra:mm:ii}. Employing that
 for all $\TrSo\in\cwrSo$ it holds $K^{\OpSy}\oeRa\geq 16
 K[\gb_{\Di}\vee\ObSoNoL {\Di}\oEvs[\Di]]$ for all $\moeDi\leq
 \Di\leq\poeDi$ it follows that $K_\ObSoNoL\oeRa/[\gb_{\Di}\vee\ObSoNoL {\Di}
\oEvs[\Di]]\geq  16 c_2 K$ where $c_2:= K_{\ObSoNoL}/K^{\OpSy}\geq 12$ for
all $\ObSoNoL\in(0,\widetilde{\ObSoNoL}_{\OpSy})$ since $K_{\ObSoNoL}\to\infty$ as
$\ObSoNoL\to0$. Therefore, by applying
\eqref{ad:th:ra:mm:pr:e2} we have
\begin{equation*}
\sum_{\Di=\moeDi}^{\poeDi}\Ex_{\TrSo} P_{\DiRvSo|\ObSo}\big(
\HnormV{\DiRvSo[\Di]-\TrSo}^2> {K_{\ObSoNoL}}\oeRa\big)
\leq 4 \exp(-K_{\epsilon}/[12K^{\OpSy}]).
\end{equation*}
and hence from Lemma \ref{ad:le:opm} follows  for all
$\ObSoNoL\leq (\widetilde{\ObSoNoL}_{\OpSy}\wedge\ObSoNoL_{\OpSy})$
\begin{multline*}
\Ex_{\TrSo}P_{\RvDiRvSo|\ObSo}(
\HnormV{\RvDiRvSo-\TrSo}^2> K_{\ObSoNoL}\oeRa)
\\\leq 4\exp\big(-\oeDi\{C_{\Ev}/5-\log\DiMa/\oeDi\}\big) + 4\exp(-K_{\ObSoNoL}/[12K^{\OpSy}]).
\end{multline*}
Observe, that  $(\widetilde{\ObSoNoL}_{\OpSy}\wedge\ObSoNoL_{\OpSy})$
depends only on the class $\cwrSo$ and thus the upper bound given in the last display holds true
uniformly for all $\TrSo\in\cwrSo$,
which  implies the assertion
\ref{ad:th:ra:mm:ii} by using that $K_{\ObSoNoL}\to\infty$, $\oeDi\to\infty$ and $\log \DiMa/\oeDi=o(1)$
as $\ObSoNoL\to0$, and completes the proof.
\end{proof}
\subsection{Proof of Theorem \ref{ad:th:be:mm}}
\begin{lem}\label{ad:le:be:mm:rem}
If Assumption \ref{tr:as:pv} and \ref{ad:as:ev}  hold true then for all
$\TrSo\in\cwrSo$ and  $\ObSoNoL\in(0,\ObSoNoL_o)$
\begin{enumerate}[label=\emph{\textbf{(\roman*)}},ref=\emph{\textbf{(\roman*)}}]
\item\label{ad:le:be:mm:rem:i} $\sum_{j=1}^{\DiMa}\sigma_j^2\Ev_j^2\ObSoNoL^{-2}\Ex_{\TrSo}\{(\ObSo_j-\Ev_j\TrSo[j])P_{\RvDi|\ObSo}(j\leq
\RvDi\leq\DiMa)\}^2$
\\\null\hfill
$\leq\ObSoNoL \poeDi \oEvs[\poeDi]+ 10\Evs_1\exp\big(-\oeDi/5+2\log\DiMa\big)$;
\item\label{ad:le:be:mm:rem:ii}
  $\sum_{j=1}^{\DiMa}(\RvSoEr_j-\TrSo[j])^2\Ex_{\TrSo}\Ex_{\RvDi|\ObSo}\{\Ind{\{1\leq\RvDi<j\}}
  + (\sigma_j^2\RvSoVa_j^{-1})^2\Ind{\{j\leq\RvDi\leq \DiMa\}}\}+
  \sum_{j>\DiMa}(\RvSoEr_j-\TrSo[j])^2$\\\null\hfill$\leq
  \gb_{\moeDi} +\HnormV{\RvSoEr-\TrSo}^2\{d^{-2}\ObSoNoL\mEvs[\moeDi]+ 2\exp\big(-\frac{C_{\Ev}(1\vee\rSo)}{5}\oeDi+\log\DiMa\big)\}$.
\end{enumerate}
\end{lem}
\begin{proof}[\noindent{\color{darkred}\sc Proof of Lemma
    \ref{ad:le:be:mm:rem}.}]\label{ad:le:be:mm:rem:pr}
The proof follows along the lines of the proof of Lemma
\ref{ad:le:bo:rem}, where we replace
  $\meDi$, $\peDi$, $\treDi$ and $\treRa$ by its counterpart
  $\moeDi$, $\poeDi$, $\oeDi$ and $\oeRa$, respectively.

Consider \ref{ad:le:be:mm:rem:i}. 
Following the proof of \eqref{ad:le:bo:rem:pr:e1} it is
straightforward to see that
\begin{multline}\label{ad:le:be:mm:rem:pr:e1}
\sum_{j=1}^{\DiMa}\sigma_j^2\Ev_j^2\ObSoNoL^{-2}\{(\ObSo_j-\Ev_j\TrSo[j])P_{\RvDi|\ObSo}(j\leq
\RvDi\leq\DiMa)\}^2
\\
\leq \sum_{j=1}^{\poeDi}\ObSoNoL\Evs_j\ObSoNo^2_j +
\sum_{j=1}^{\DiMa}\ObSoNoL\Evs_j\ObSoNo_j^2P_{\RvDi|\ObSo}(\poeDi<
\RvDi\leq\DiMa)
\end{multline}
and following line by line the proof of \eqref{ad:le:bo:rem:pr:e2} we
conclude
\begin{multline}\label{ad:le:be:mm:rem:pr:e2}
\sum_{j=1}^{\DiMa}\ObSoNoL\Evs_j\Ex_{\TrSo}\{\ObSoNo_j^2 P_{\RvDi|\ObSo}(\poeDi<\RvDi\leq\DiMa)\}\\
\leq 6\Evs_1\exp(-\DiMa/6)+2\Evs_1\DiMa\Ex_{\TrSo}P_{\RvDi|\ObSo}(\poeDi<\RvDi\leq\DiMa).
\end{multline}
We distinguish two cases. First, if $\poeDi=\DiMa$, then assertion \ref{ad:le:be:mm:rem:i} follows  by combining
\eqref{ad:le:be:mm:rem:pr:e1} and $\Ex_{\TrSo}\ObSoNo_j^2=1$. Second, if
$\poeDi<\DiMa$, then the definition \eqref{ad:de:omp} of $\poeDi$ implies
$\DiMa>5\oeDi$ which in turn implies the assertion \ref{ad:le:be:mm:rem:i} by combining
\eqref{ad:le:be:mm:rem:pr:e1}, $\Ex_{\TrSo}\ObSoNo_j^2=1$,  \eqref{ad:le:be:mm:rem:pr:e2} and Lemma \ref{ad:le:opm} \ref{ad:le:opm:i}. Consider
\ref{ad:le:bo:rem:ii}.
Following the proof of Lemma \ref{ad:le:bo:rem} \ref{ad:le:bo:rem:ii}
we obtain
\begin{multline*}
\sum_{j=1}^{\DiMa}(\RvSoEr_j-\TrSo[j])^2\Ex_{\TrSo}\{\Ind{\{1\leq\RvDi<j\}}
  + (\sigma_j/\RvSoVa_j)^2\Ind{\{j\leq\RvDi\leq \DiMa\}}\}+ \sum_{j>\DiMa}(\RvSoEr_j-\TrSo[j])^2\\
\leq \HnormV{\RvSoEr-\TrSo}^2\{\Ex_{\TrSo}P_{\RvDi|\ObSo}(1\leq\RvDi<\moeDi)+
d^{-2}\ObSoNoL\mEvs[\moeDi]\}+\sum_{j>\moeDi}(\RvSoEr_j-\TrSo[j])^2
\end{multline*}
The assertion \ref{ad:le:be:mm:rem:ii} follows now by combining
the last estimate and Lemma \ref{ad:le:opm} \ref{ad:le:opm:ii}, which
completes the proof.
\end{proof}
\begin{proof}[\noindent{\color{darkred}\sc Proof of Theorem
    \ref{ad:th:be:mm}.}]\label{ad:th:be:mm:pr}
The proof follows line by line the proof of Theorem \ref{ad:th:bo}
using Lemma \ref{ad:le:be:mm:rem} rather than Lemma
\ref{ad:le:bo:rem}, more precisely from Lemma \ref{ad:le:be:mm:rem} follows
  \begin{multline*}
   \Ex_{\TrSo} \HnormV{\hSo-\TrSo}^2\leq \sum_{j=1}^{\DiMa}2\sigma^2\Ev^2_j\ObSoNoL^{-2}\Ex_{\TrSo}\{(\ObSo_j-\Ev_j\TrSo[j])P_{\RvDi|\ObSo}(j\leq \RvDi\leq
  \DiMa)\\ + \sum_{j=1}^{\DiMa}2(\RvSoEr_j-\TrSo[j])^2\Ex_{\TrSo}\{(\sigma_j/\RvSoVa_j)P_{\RvDi|\ObSo}(j\leq \RvDi\leq
  \DiMa)+
  P_{\RvDi|\ObSo}(1\leq\RvDi<j)\}^2+\sum_{j>\DiMa}(\RvSoEr_j-\TrSo[j])^2\\
\leq 2\{\ObSoNoL \poeDi \oEvs[\poeDi]
+10\Evs_1\exp\big(-\frac{C_{\Ev}(1\vee\rSo)}{5}\oeDi+2\log\DiMa\big)\}\\
+ 2\{\gb_{\moeDi} +\HnormV{\RvSoEr-\TrSo}^2\{d^{-2}\ObSoNoL\mEvs[\moeDi]+ 2\exp\big(-\frac{C_{\Ev}(1\vee\rSo)}{5}\oeDi+\log\DiMa\big)\}\}.
 \end{multline*}
Taking  further into account the definition \eqref{ad:de:omp} of $\moeDi$ and
$\poeDi$, we have  $\gb_{\moeDi}\leq 8L_\Ev
C_{\Ev}(1+1/d)(1\vee\rSo)\oeRa$   and (keeping in mind Assumption \ref{tr:as:mi})
$\poeDi \leq 
D^\star\oeDi$
with $D^\star:=D^\star(\cwrSo,\Ev):=\ceil{5 L_\Ev(1\vee\rSo)/\kappa}$,
which in turn implies $\ObSoNoL \poeDi \oEvs[\poeDi]\leq L_\Ev D^\star\mEvs[D^\star]\oeRa$, while
trivially
$\ObSoNoL\mEvs[\moeDi]\leq\ObSoNoL\mEvs[\oeDi]\leq\oeRa$ and $\HnormV{\RvSoEr-\TrSo}^2\leq\rSo$. By
combination of these estimates we obtain uniformly for all $\TrSo\in\cwrSo$  that
  \begin{multline*}
   \Ex_{\TrSo} \HnormV{\hSo-\TrSo}^2\leq \{2L_\Ev
   D^\star\mEvs[D^\star]+16 L_\Ev C_{\Ev}(1+1/d)(1\vee\rSo)+2
   d^{-2}\rSo\}\oeRa\\
+(20\Evs_1+ 4\rSo)\exp\big(-\frac{C_{\Ev}(1\vee\rSo)}{5}\oeDi+2\log\DiMa-\log\oeRa\big)\}\oeRa.
 \end{multline*}
Note that in the last display the multiplicative factors of $\oeRa$ depend only on the class
$\cwrSo$, the constant $d$ and the sequence $\Ev$.   Thereby, the
assertion of the theorem follows from $\log(\DiMa/\oeRa)/\oeDi \to 0$ as $\ObSoNoL\to0$ which completes the proof.
\end{proof}
